\documentclass[11pt]{article}

\usepackage{verbatim}
\usepackage{float}

\makeatletter

\usepackage[dvipsnames]{xcolor}

  \usepackage{makecell}

\usepackage[T1]{fontenc}

\usepackage[margin=3cm]{geometry}
\usepackage{amsfonts,amssymb,stmaryrd,amsthm,amsmath}

\usepackage{graphicx,tikz,tikz-cd, tabularx}
\usepackage{quiver}
\usepackage{mathtools}

\usepackage{hyperref}

\usepackage[utf8]{inputenc}

\newtheorem{prop}{Proposition}[subsection]
\newtheorem{theo}[prop]{Theorem}
\newtheorem{lemma}[prop]{Lemma}
\newtheorem{coro}[prop]{Corollary}

\theoremstyle{definition}
\newtheorem{definition}[prop]{Definition}
\newtheorem{remark}[prop]{Remark}
\newtheorem{example}[prop]{Example}
\newtheorem{conjecture}[prop]{Conjecture}
\usepackage{soul}

\makeatletter
\renewcommand\@biblabel[1]{#1}
\makeatother
\newcommand\DynkinNodeSize{2mm}
\newcommand\DynkinArrowLength{3mm}
\tikzset{
  dnode/.style={
    circle,
    inner sep=0pt,
    minimum size=\DynkinNodeSize,
    fill=white,
    draw},
  middlearrow/.style={
    decoration={markings,
      mark=at position 0.6 with
      {\draw (0:0mm) -- +(+135:\DynkinArrowLength); \draw (0:0mm) -- +(-135:\DynkinArrowLength);},
    },
    postaction={decorate}
  },
  leftrightarrow/.style={
    decoration={markings,
      mark=at position 0.999 with
      {
      \draw (0:0mm) -- +(+135:\DynkinArrowLength); \draw (0:0mm) -- +(-135:\DynkinArrowLength);
      },
      mark=at position 0.001 with
      {
      \draw (0:0mm) -- +(+45:\DynkinArrowLength); \draw (0:0mm) -- +(-45:\DynkinArrowLength);
      },
    },
    postaction={decorate}
  },
  sedge/.style={
  },
  dedge/.style={
    middlearrow,
    double distance=0.5mm,
  },
  tedge/.style={
    middlearrow,
    double distance=1.0mm+\pgflinewidth,
    postaction={draw}, % third line
  },
  infedge/.style={
    leftrightarrow,
    double distance=0.5mm,
  }
}

\newcommand\revddots{\mathinner{\mkern1mu\raise\p@\vbox{\kern7\p@\hbox{.}}\mkern2mu\raise4\p@\hbox{.}\mkern2mu\raise7\p@\hbox{.}\mkern1mu}}
\makeatletter
\def\revddots{\mathinner{\mkern1mu\raise\p@\vbox{\kern7\p@\hbox{.}}\mkern2mu\raise4\p@\hbox{.}\mkern2mu\raise7\p@\hbox{.}\mkern1mu}}
\makeatother

\newcommand\bmat{\begin{pmatrix}}
\newcommand\emat{\end{pmatrix}}
\newcommand\longto{{\, \longrightarrow \, }}

\newcommand\CC{\mathbb{C}}
\newcommand\NN{\mathbb{N}}\newcommand\ZZ{\mathbb{Z}}

\newcommand\inv{{^{-1}}}

    \newcommand\Bc{{\mathcal B}}

\newcommand\base{{\mathcal B}}
\newcommand\diag{{\operatorname{diag}}}

\newcommand\Tr{{\operatorname{Tr}}}

\newcommand\hy{{\widehat y}}\newcommand\halpha{{\widehat \alpha}}\newcommand\hv{{\widehat v}}
\newcommand\hbeta{{\widehat \beta}}

\newcommand\hPhi{{\widehat\Phi}}
\newcommand\hw{{\widehat w}}\newcommand\hW{{\widehat W}}

\newcommand\hG{{\widehat G}}
\newcommand\hB{{\widehat B}}
\newcommand\hT{{\widehat
    T}}
    \newcommand\hU{{\widehat U}} 
    \newcommand\hDel{{\widehat \Delta}} 
    \newcommand{\hg}{\widehat g}

\newcommand{\ha}{{\widehat \alpha}}
\newcommand{\hb}{\widehat \beta}

\newcommand{\wh}{\widehat}

\newcommand{\Bgg}{{\mathbf{B}}}
\newcommand{\ngg}{{\mathbf{n}}}
\newcommand{\pgg}{{\mathbf{p}}}

\newcommand\rddots
   {\mathinner{\mkern1mu
       \raise 1pt\hbox{.}\mkern2mu
       \raise 4pt\hbox{.}\mkern2mu
       \raise 7pt\hbox{.}\mkern1mu}}

\newcommand\SL{\operatorname{SL}}\newcommand\Sp{\operatorname{Sp}}
\newcommand\Spin{\operatorname{Spin}}

\newcommand\Orb{{\mathcal O}}

\renewcommand\sl{{\mathfrak{sl}}}
\renewcommand{\lg}{{\mathfrak g}}\newcommand{\hlg}{{\widehat{\mathfrak g}}}

\newcommand\rd{{\operatorname{rd}}}

\newcommand{\lif}{\upharpoonleft \kern-0.35em}
\newcommand{\spc}{\kern0.2em}
\newcommand{\ii}{\mathbf{i}}
\newcommand{\jj}{\mathbf{j}}
\newcommand{\sba}{\overline{s}_\alpha}

\newcommand{\pia}{\varpi_\alpha}
\newcommand{\pib}{\varpi_\beta}
\newcommand{\piha}{\varpi_\halpha}

\newcommand{\sphweight}{\Lambda^+_{G/\hG}}

\newcommand{\delamw}{\Delta^\lambda_w}

\title{Generalised spherical minors and their relations}
\author{Luca Francone}
\begin{document}

\maketitle

\begin{abstract}
    Let $\hG$ be a spherical subgroup of minimal rank of the semisimple simply connected complex algebraic group $G$. 
We define some functions on the homogeneous space $G/\hG$ that we call \textit{generalised spherical minors}. When $G= \hG \times \hG$, we recover Fomin-Zelevinsky generalised minors.
We prove that generalised spherical minors satisfy some integer coefficients polynomial relations, that extend \cite{fomin1999double}[Theorem 1.17] and that have the shape of exchange relations of LP-algebras.
\end{abstract}

\tableofcontents

\section{Introduction }
Let $G$ be a  complex connected reductive algebraic group
  and $\hG$ be a connected reductive subgroup of $G$. Let $Y$ be the complete flag variety of $G.$
   We assume that the pair $(G,\hG)$ is \textit{spherical of minimal rank}. That is: there exists a point $y \in Y$ such that 
    \begin{enumerate}
        \item The $\hG$-orbit $\hG \cdot y $, of $y$, is open in $Y$;
        \item The stabiliser $\hG_y$ of $y$, in $\hG$, contains a maximal torus of $\hG$.
    \end{enumerate}
An important example is when $\hG$ is diagonally embedded in $\hG \times \hG$. More generally, the classification of spherical pairs of minimal rank has been done in \cite{spherangmin} and ultimately relies on the following theorem.

\begin{theo}
\label{theo:list sph rank min}
    Assuming  that $\hG$ is simple and $G$ is semisimple simply connected, the complete list of spherical pairs of minimal rank is:
\begin{enumerate}
\item $\hG=G$;
\item \label{list:tensor} $\hG$ is simple, simply connected and diagonally embedded in $G=\hG\times \hG$;
\item \label{list:slsp}
$(\SL_{2n},\Sp_{2n})$ with $n\geq 2$;
\item \label{list:spin}
$(\Spin_{2n},\Spin_{2n-1})$ with $n\geq 4$;

%??? Is it $SO_{2n-1}$ or Spin$_{2n-1}$ ???
\item \label{list:G2}
$(\Spin_7,G_2)$;
\item \label{list:F4}
$(E_6,F_4)$.
\end{enumerate}
\end{theo}

%\marginnote{\color{red} serve sta frase?}
\noindent Our present work on spherical pairs of minimal rank aims to generalise some results and ideas of \cite{fomin1999double} in the context of spherical pairs of minimal rank. In order to present our  results, we introduce some notation.
%in sight of a systematic study of total positivity, Lusztig zones \cite{lusztig2022total} and their relation with the branching problem of the pair $(G,\hG).$

\bigskip

 Assume from now on that $G$ is semisimple and simply connected.  Let $B$ be a Borel subgroup of $G$, $T$ be a maximal torus of $B$ and  $W$ be the Weyl group of $G$ with respect to the torus $T$.

 \noindent We define the sets
$$
Z:= \{ z \in W \, : \, Bz\hG \quad \text{is open in} \, \, G \} \hspace{0.8 cm} Z_m:= \{ z \in Z \, : \, \ell(z) \, \, \text{is minimal} \},
$$
where $\ell$ denotes the length of the Coxeter group $W$, with respect to the set of simple roots determined by $B$. The sets $Z$ and $Z_m$ are non-empty. Let $z \in Z_m$ and set $y=z^{-1}.$

 We denote by $X(T)$ (resp. $X(T)^+$) the set of characters (resp. dominant characters) of $T$. For $\lambda \in X(T)^+$, the notation $V(\lambda)$ stands for the irreducible $G$-representation of highest weight $\lambda$.
%\subsubsection{On objective 2: generalised spherical minors}
 Let

\begin{equation}
\Lambda_{G/\hG}^+= \{ \lambda \in X(T)^+ \, : \, V(\lambda)^\hG \neq 0 \}
\end{equation}
be the monoid of \textit{spherical weights}. As a monoid,   $\Lambda_{G/\hG}^+$ is freely generated by a unique set of free generators, denoted by $\Bc$. For the pairs of Theorem \ref{theo:list sph rank min}, the set $\Bc$ is described in \eqref{eq:sph base of weights} below.

\bigskip

For any $w \in W$ and $\lambda \in \Lambda_{G/\hG}^+$, we define a function $\Delta^\lambda_w$ (Definition \ref{def:sph minors}) on the homogeneous spaces $G/\hG$. Such a function is called \textit{generalised spherical minor}. After proving some general properties of generalised spherical minors, we focus on algebraic relations between them. Before stating our main result, we make some observations.

Because of Lemma \ref{lem:irred generalised spherical minors}, any generalised spherical minor is actually a monomial in the minors of the form $\Delta^b_{w}$, with $b \in \Bc$ and $w \in W$. 
Thus, algebraic relations between minors can be expressed in terms of the functions $\Delta^b_w$, with $b \in \Bc$ and $w \in W$.

\noindent Moreover, the definition of the functions $\Delta^\lambda_w $ depends, in principle, on the choice of the element $z$. By Lemma \ref{lem:sph minors independent z}, the set of generalised minors only depends on $z$ up to some signs.
Thus, the choice of the element $z \in Z_m$ doesn't substantially impact the set of relations.

\noindent If $\hG$ is not simple, by \cite{spherangmin}[Proposition 3.2] there exists two spherical pairs of minimal rank $(G_1, \hG_1)$ and $ (G_2, \hG_2)$ such that 
$ (G, \hG)= (G_1 \times G_2, \, \hG_1 \times \hG_2).$
Moreover, Lemma \ref{lem:characterisation generalised spherical minors} implies that the generalised spherical minors of the pair $(G, \hG)$, are the pullback of the generalised spherical minors of the two smaller pairs, via the natural projections $G/\hG \longto G_1/\hG_1$ and $G/\hG \longto G_2/\hG_2.$ 
Thus, algebraic relations between generalised spherical minors of the pair $(G,\hG)$ are determined by the relations between the minors of the two factors.

\bigskip

\noindent For the rest of this introduction, we assume that $\hG$ is simple. In Section \ref{sec:fund identities}, we associate to the pair $(G,\hG)$ an element $z \in Z_m$ and an integer coefficients matrix $A=(a_{b,b'}) \in \ZZ^{\Bc \times \Bc}$.  Moreover, to any  $b \in \Bc$, we attach an element $\hw_ b \in W$ and a set $U_b \subseteq W$ such that $e \in U_b$.  Then, we prove the following theorem.
%\marginnote{\color{red} ocho che in origine era $\hw_b \in \hW$}

\begin{theo}
    \label{thm:fundamental identities sph mnors}
   Let $b \in \Bc$ and $w \in W$. If the conditions 
   $$\ell(y\inv w u y)= \ell( y \inv w y) + \ell(y\inv u y) \quad \text{and} \quad \ell( y\inv w u \hw_b y)= \ell( y\inv w y ) + \ell(y\inv u \hw_b y)$$
   are satisfied  for any $u \in U_b$,
then the following identity holds:

   \begin{equation}
   \label{eq:fundamental iddentity sh minors}
      \Delta^b_{w} \Delta^b_{w \hw_b } =  \sum_{u \in U_b \setminus \{e\}} (-1)^{\ell(y\inv u y)-1} \Delta^b_{wu} \Delta^b_{w u \hw_b} + \epsilon_b \prod_{b' \in \Bc \setminus \{b\}} \biggl( \Delta^{b'}_{w} \biggr)^{-a_{b,b'}},
       \end{equation}
       for some $\epsilon_b \in \{\pm1\}$, which only depends on $b$.
\end{theo}

 The matrix $A$, the elements $\hw_b$ and the sets $U_b$ are defined case by case, accordingly to the list of pairs of Theorem \ref{theo:list sph rank min}. A pair independent description of the previous data would be of great interest, but it is still missing . 
 The only exception concerns the matrix $A$. Indeed, it agrees with the Cartan matrix of the small root system of the pair $(G, \hG)$. Still, we have no conceptual explanation of this fact. Moreover, we have the following conjecture.

\begin{conjecture}
    \label{conj: sign epsilon_b}
    In the notation of Theorem \ref{thm:fundamental identities sph mnors}, we have that $\epsilon_b= (-1)^{|U_b|}.$
\end{conjecture}

 \noindent When $G=\hG \times \hG$, by Lemma \ref{lem:sph minors tensor prod equals FZ minors}, generalised spherical minors recover Fomin-Zelevinsky generalised minors. The relations described in Theorem \ref{thm:fundamental identities sph mnors} specialize to \cite{fomin1999double}[Theorem 1.17].
 The relations provided by \cite{fomin1999double}[Theorem 1.17] stand at the very base of all the known cluster structures 
related to algebraic groups, such as the one considered in 
\cite{berenstein2005cluster3} \cite{geiss2008partial}, \cite{geiss2011kac}, \cite{goodearl2021integral}, \cite{galashin2023braid}, \cite{fei2016tensor}, \cite{francone2023minimal}.
Moreover, together with the companion relations provided by \cite{fomin1999double}[Theorem 1.16], they allow to unravel deep links between double Bruhat cells, total positivity, the combinatorial properties of the Luszting canonical basis and the Lusztig zones. These deep connections ultimately led to the construction of a class of polyhedral models for tensor product multiplicities known as $\ii$-trails models. See \cite{berenstein1999tensor} and \cite{zelevinsky2001littlewood}.
 %Thus, Theorem \ref{thm:fundamental identities sph mnors} suggests many directions in which some of the pioneering works of Berensetein, Fomin and Zelevinsky could be generalised. 
 We observe that the relations described in Theorem \ref{thm:fundamental identities sph mnors} always have the shape of exchange relations of LP-algebras  in the sense of \cite{lam2016laurent}. Nevertheless, they are  manifestly written in the form of exchange relations of cluster and generalised cluster algebras only when $G= \hG \times \hG$. Indeed, in any other case, the right hand side of \eqref{eq:fundamental iddentity sh minors} consists of at least 3 monomials in some generalised spherical minors.
 We hope that our result could be used to construct LP-algebras structures, in the context of spherical pairs of minimal ranks, generalising some of the well known cluster algebras structures previously mentioned. 
 We also hope that Theorem \ref{thm:fundamental identities sph mnors} could shed a light on some interpretation of the combinatorial properties of the Lusztig zones, as defined in \cite{lusztig2022total}.

\section{Complements on spherical subgroups of minimal rank}

\subsection{Root systems and Weyl groups}
\label{sec: reminers sph rank min}

Most of the contents of this section is adapted from \cite{spherangmin} and \cite{francone2023multiplicity}. We worn the reader that the present notation exchanges the role of $G$ and $\hG$ with respect to the one used in \cite{francone2023multiplicity}. 

\bigskip

We fix respectively a  maximal torus and Borel subgroup $\hT$ and $\hB$ of $\hG$. 
We make the assumption that $T \cap \hG= \hT$ and $B \cap \hG=\hB$.
This implies that $\hT \subseteq \hB$. 
Let $U$ and $\hU$ denote the unipotent radicals of $B$ and $\hB$ respectively.
The notation $X(\hT)$ stands for the character group of $\hT.$
Moreover, denote by $\Phi$  (resp. $\hPhi$)  the root system of $G$ (resp. $\hG$). 
We denote by $\Phi^+$ (resp. $\hPhi^+$) the set of positive roots determined by $B$ (resp. $\hB$). Moreover we set $\Phi^-:= - \Phi^+$ and $\hPhi^-:= - \hPhi^+$.
%Let $W$ (resp. $\hW)$ be the Weyl group of $G$ (resp. $\hG$).
%Let $B^-$ be the opposite Borel subgroup of $B$, with respect to $T$, and $U^-$ be its unipotent radical.

Let $\rho : X(T) \longto X(\hT)$ the restriction map. By \cite{spherangmin}[Lemma 4.2], we have that $\rho$ maps $\Phi$ to $\hPhi$.
Therefore, we can define $\overline{ \rho}: \Phi \longto \hPhi$ as the restriction of $\rho$. 
Let 
$$
\hPhi^1:= \{ \hbeta \in \hPhi \, : \, |\bar \rho\inv(\hbeta)|= 1 \} \quad  
\hPhi^2:= \{ \hbeta \in \hPhi \, : \, |\overline{\rho}\inv(\hbeta)|= 2 \}. 
$$
By \cite{spherangmin}[Lemma 4.4], we have that $\hPhi= \hPhi^1 \cup \hPhi^2.$ 
Hence $\Phi= \Phi^1 \cup \Phi^2$, where $\Phi^1= \bar \rho \inv( \hPhi^1)$ and $\Phi^2= \bar \rho \inv( \hPhi^2)$. 
Moreover, by \cite{spherangmin}[Lemma 4.6] we have that $\bar \rho \inv( \hDel)= \Delta$. We use the notation
$$
\hDel^{i}:=\hPhi^i \cap \hDel  \quad \text{and} \quad \Delta^i= \Phi^i \cap \Delta \quad \text{for} \, \, i=1,2.
$$
Observe that, in cases 3 to 6 of Theorem \ref{theo:list sph rank min}, we have that $\hPhi^1= \hPhi^l$ and $\hPhi^2=\hPhi^s$ where $\hPhi^l$ (resp. $\hPhi^s$) denotes the set of long (resp. short) roots. Finally, if $\hb \in \hPhi^2$, then $\bar \rho \inv(\hb)$ consists of two strongly orthogonal roots.

\bigskip

Let $\hW$ be the Weyl group of $\hG$ with respect to $\hT$. 
By \cite{Br:ratsmooth}[Lemma 2.3], we have that $\hT$ is a regular torus of $G$. Hence $\hW$ is a subgroup of $W$. 
As a consequence of \cite{spherangmin}[Lemma 4.1, Lemma 4.3] we have the following lemma.

\begin{lemma}
    \label{lem: reflections sph rk min}
    Let $\beta \in \Phi$ and $\hbeta= \rho(\beta)$. Let $s_\beta \in W$ (resp. $s_{\hb}\in \hW$) be the reflection associated $\beta$ (resp. $\hbeta$). The following holds.

    \begin{enumerate}
        \item If $\beta \in \Phi^1$, then $s_\beta= s_\hbeta$.
        \item If $\beta \in \Phi^2$ and $\bar \rho\inv(\hbeta)=\{ \beta, \beta'\}$, then $s_\hbeta=s_\beta s_{\beta'}= s_{\beta'}s_\beta$.
        \item The reflection $s_\beta $ belongs to $ \hW$ if and only if $\beta \in \Phi^1.$
    \end{enumerate}
\end{lemma}

Recall that a collection $\ii= (i_m, \dots, i_1)$ (resp. $\wh \ii= (\wh i_m, \dots, \wh i_1)$), of elements of $\Delta$ (resp. $\hDel$), is a  \textit{reduced expression} for $w \in W$ (resp. $\hw \in \hW$) if $s_{i_m} \cdots s_{i_1}=w$ (resp. $s_{\wh i_m} \cdots s_{ \wh i_1}= \hw)$ and $m$ is minimal. 

\begin{lemma}
    \label{lem: compatible red expressions spherical}
    Let $\hw \in \hW$ and $\widehat \ii= (\widehat i_m, \dots , \widehat i_1)$ be a reduced expression for $\hw$, considered as an element of the Coxeter group $\hW$. Let $\ii$ be the expression obtained by replacing $\widehat \ii_k$ with any ordering of the simple roots in $\bar \rho \inv(\widehat \ii_k)$, for any $k \in [m]$. Then, $\ii$ is a reduced expression for $\hw$, considered as an element of the Coxeter group $W$.
\end{lemma}
 
\begin{proof}
     Let $\widehat v \in \hW$ and $\ha \in \hDel$. Since $\bar \rho$ is $\hW$-equivariant,  $\widehat v \ha \in \hPhi^-$ if and only if $\widehat v \bar \rho\inv(\ha) \subseteq \Phi^-.$ Using that $\bar \rho \inv(\ha)$ cosists of either one root, or of two strongly orthogonal roots, the proof follows by an easy induction on $m$. 
\end{proof}

\subsection{Compatible pinnings}

By \cite{francone2023multiplicity}[Lemma 3], assigning to an element $w \in W$ the orbit $Bw\hG/\hG$ gives a bijection between the quotient $W/\hW$ and the set of $B$-orbits in $G/\hG.$ Therefore, the sets $Z$ and $Z_m$ of the introduction are non-empty.
Let $z \in Z_m$ 
and define 
$$
y=z\inv.
$$
We refer to \cite{francone2023multiplicity}[Section 4] for examples of elements of $Z_m$. (In \cite{francone2023multiplicity}[Section 4], an element belonging to $Z_m$ is denoted denoted by $\hy_0$).

By \cite{francone2023multiplicity}[Lemma 5], we have that for any $\hbeta \in \hPhi^+ \cap \hPhi^2$ there exists exactly one root in $\bar \rho^{-1}(\hb)$ which is sent to $\Phi^-$ by the action of $z$. Therefore, we adopt the following notation for labelling roots.
\begin{enumerate}
    \item[-] If $\hbeta \in \Phi^1 \cap \Phi^+$, then $\bar \rho \inv(\hbeta)= \{ \beta \}$.
     \item[-] If $\hbeta \in \Phi^2 \cap \Phi^+$, then $\bar \rho \inv(\hbeta)= \{ \beta_\pgg, \beta_\ngg\}$ where $ z \cdot \beta_\ngg\in \Phi^-$.
\end{enumerate}
We stress that this notation depends on the choice of the element $z \in Z_m$ and it is therefore only used when such an element is fixed or sufficiently clear from the context.

\bigskip

Let $\lg$ be the Lie algebra of $G$. If $\beta \in \Phi$, we denote by $\lg_\beta$ the corresponding root space. Moreover, a triple $(X_\beta, H_\beta, X_{-\beta})$ consisting of elements $X_{\pm \beta} \in \lg_{\pm \beta}$ such that $[X_\beta, X_{-\beta}]= H_\beta$ and $[H_\beta, X_{\beta}]=2X_\beta$ is said to be an \textit{$\sl_2$-triple} for $\beta$. 
A similar notation, topped by a hat symbol, is used relatively to $\hG$.
\begin{definition}
\label{def:compatible sl2 triples}
    Let $\hbeta \in \hPhi^+$ and $(X_\hbeta, H_\hbeta, X_{-\hbeta})$ be an $\sl_2$-triple for $\hbeta$.
\begin{itemize}
    \item[-] If $\hbeta \in \hPhi^1$, an $\sl_2$-triple $(X_\beta, H_\beta, X_{-\beta})$ for $\beta$  is \textit{compatible} if
    $$
    X_\beta=X_\hbeta \quad H_\beta=H_\hbeta \quad  X_{-\beta}= X_{-\hbeta}.
    $$ 
    \item[-] If $\hbeta \in \hPhi$, two $\sl_2$ triples $(X_{\beta_\ngg}, H_{\beta_\ngg}, X_{-\beta_\ngg})$ and $(X_{\beta_\pgg}, H_{\beta_\pgg}, X_{-\beta_\pgg})$ for $\beta_\ngg$ and $\beta_\pgg$ respectively are said to be compatible if 
    $$
    X_{\beta_\ngg}+ X_{\beta_\pgg}= X_\hbeta \quad H_{\beta_\ngg}+ H_{\beta_\pgg}= H_\hbeta \quad  X_{-\beta_\ngg}+X_{-\beta_\pgg} =X_{-\hbeta}.
    $$
\end{itemize}
\end{definition}

By \cite{francone2023multiplicity}[Lemma 9], for any $\hb \in \hPhi^+$, any $\sl_2$-triple for $\hbeta$  admits compatible $\sl_2$-triples, which are obviously unique. 

From now on, for any simple root $\ha \in \hDel $, we fix an $\sl_2$-triple $(X_\ha, H_\ha, X_{-\ha})$ for $\ha$ and the corresponding compatible $\sl_2$-triples in $\lg$, for which we use the notation of Definition \ref{def:compatible sl2 triples}. We refer to the previous data as a choice of a \textit{compatible pinning} for the pair $(G, \hG)$.

We use the notation
$$
x_{\pm \ha}(t):= \exp(t X_{\pm \ha}), \quad x_{\pm \alpha}(t):= \exp(t X_{\pm \alpha}) \qquad (\ha \in \hDel, \, \, \alpha \in \Delta, \, \, t \in \CC). 
$$
Moreover, we denote by $\alpha^\vee: \CC^* \longto T$ (resp. $\ha^\vee : \CC^* \longto \hT$) the one parameter subgroup corresponding to the root $\alpha \in \Delta$ (resp. $\ha \in \hDel$).

\noindent Note that, if $t \in \CC^*$ and $\halpha \in \hDel^1$, then

\begin{equation}
\label{eq:ops sph rk min type 1}
    x_\halpha(t)=x_\alpha(t) \quad \halpha^\vee(t)=\alpha^\vee(t) \quad x_{-\halpha}(t)=x_{-\alpha}(t).
\end{equation}
Moreover, if $t \in \CC^*$ and $\halpha \in \hDel^2$, then

\begin{equation}
\label{eq:ops sph rk min type 2}
  x_\halpha(t)=x_{\alpha_\ngg}(t)x_{\alpha_\pgg}(t) \quad \halpha^\vee(t)=\alpha_\ngg^\vee(t)\alpha_\pgg^\vee(t) \quad x_{-\halpha}(t)=x_{-\alpha_\ngg}(t)x_{-\alpha_\pgg}(t).
\end{equation}
because of \cite{francone2023multiplicity}[Lemma 9].
We conclude this section with a useful lemma and an example.

\begin{lemma}
    \label{lem:G simply connected, Gh simply connected}
    If $G$ is semisimple (resp. semisimple and simply connected), so is $\hG$.
\end{lemma}

\begin{proof}
 If $G$ is semisimple, then $\hG/ (Z(G) \cap \hG)$ is semisimple because of \cite{spherangmin}[Proposition 3.2]. Hence $\hG$ is semisimple since $Z(G) \cap \hG$ is finite. 
 
 Next suppose that $G$ is semisimple and simply connected.
    Let $\beta \in \Delta$ and $\halpha \in \hDel$. If $G$ is simply connected, then $\pib \in X(T)$. By Formulas \eqref{eq:ops sph rk min type 1} and \eqref{eq:ops sph rk min type 2} we have that 

      $$ \langle \rho(\pib) , \halpha^\vee \rangle = \delta_{\rho(\beta), \halpha}. $$
 Hence $\rho(\pib)=\varpi_{\rho(\beta)}$. Since $\rho : \Delta \longto \hDel$ is surjective, we deduce that for any $\halpha \in \Delta$,  $\varpi_\ha \in X(\hT)$. This implies that $\hG$ is simply connected.
\end{proof}

\begin{coro}[of the proof] \label{cor:rho pia sphr} For any $\alpha \in \Delta$, we have $\rho(\pia)=\varpi_{\rho(\alpha)}.$
    \end{coro}

\begin{example}
\label{ex: sp sl pt 1} 
We start discussing the running example of the pair$(\SL_{2n}, \Sp_{2n})$.
Fix $n\geq 2$. 
Let $V$ be a $2n$-dimensional vector space 
with fixed basis $\base=(e_1,\dots,e_{2n})$ and $G$ be the special linear group of $V$. Let $T$ (resp $B$) be the maximal torus (resp. Borel subgroup) of $G$ consisting in diagonal (resp. upper triangular)
matrices with respect to this basis. For $ i \in [1,2n]$, let $ \varepsilon_i $ be the character of $T$ corresponding to the diagonal entry in position $i$. We have that $ X(T)=\oplus_{i=1}^{2n}\ZZ \varepsilon_i/(\varepsilon_1+\cdots+\varepsilon_{2n})$ and, with little abuse of notation: 

$$
\begin{array}{l}
  \Phi^+=\{ \varepsilon_i -  \varepsilon_j\,:\,1\leq i<j\leq 2n\}\\
\Delta =\{\alpha_i= \varepsilon_i- \varepsilon_{i+1} \, : \, 1 \leq i \leq 2n-1 \}\\
X(T)^+=\{\sum_{i=1}^{2n}\lambda_i  \varepsilon_i\,:\,
  \lambda_1\geq\cdots\geq \lambda_{2n} = 0\ \}.
\end{array}
$$ 

Consider the following matrices
\begin{eqnarray}
  \label{eq:defJn}
J_n=\left(
  \begin{array}{ccc}
    &&1\\
&\rddots\\
1
  \end{array}
\right);
\qquad{\rm and}\qquad
\omega=\left(
  \begin{array}{cc}
   0 &J_n\\
-J_n&0
  \end{array}
\right).\label{eq:defsympform}
\end{eqnarray}
of size $n\times n$ and $2n\times 2n$ respectively.
We consider $\omega$ as a symplectic bilinear form of $V$.  Let $\hG$ be the associated symplectic group.
Set $\hT=\{\diag(t_1,\dots,t_{n},t_{n}^{-1},\dots,t_1^{-1})\,:\,t_i\in\CC^*\}$.
Let $\widehat B$ be the Borel subgroup of $\hG$ consisting of upper triangular matrices of $\hG$.
For $i\in [1,n]$, let $\widehat \varepsilon_i$ denote the character of $T$ that maps 
$\diag(t_1,\dots,t_{n},t_{n}^{-1},\dots,t_1^{-1})$ to $t_i$; then 
$X(\hT)=\oplus_i\ZZ \wh \varepsilon_i$.
Here
$$
\begin{array}{l}
  \hPhi^+=\{ \wh \varepsilon_i\pm \wh \varepsilon_j\,:\,1\leq i<j\leq n\}\cup 
\{2 \wh \varepsilon_i\,:\,1\leq i\leq n\}, \\
\hDel=\{ \widehat \alpha_1= \wh \varepsilon_1-\wh \varepsilon_2,\,\wh \alpha_2= \wh \varepsilon_2- \wh \varepsilon_3,\dots,\,
\wh \alpha_{n-1}= \wh \varepsilon_{n-1}-\wh \varepsilon_{n},\,\wh \alpha_{n}=2 \wh \varepsilon_{n}\},
  {\rm\ and}\\
X(\hT)^+=\{\sum_{i=1}^{n}\lambda_i \wh \varepsilon_i\,:\,
  \lambda_1\geq\cdots\geq \lambda_{n}\geq 0\}.
\end{array}
$$
For $i\in [1;2n]$, set $\overline{i}=2n+1-i$. The restriction of characters is characterised by 
$$ \rho ( \varepsilon_i) = \wh \varepsilon_i \, \, \text{for} \, \, 1\leq i \leq n \hspace{ 0,3 cm} \text{and} \hspace{0,3 cm} \rho(\varepsilon_i)= - \wh \varepsilon_{ \bar i} \, \, \text{for} \, \, n < i \leq 2n. $$
Notice that $$ \hPhi^1 \cap \hPhi^+ = \{ 2 \wh \varepsilon_i \, : \, 1 \leq i \leq n \}. $$
The Weyl group $\hW$ of $\hG$ is a subgroup of the Weyl group $W=S_{2n}$ of
$G=\SL(V)$. More precisely
$$
\hW=\{\hw \in S_{2n}\,:\,\hw(\overline{i})=\overline{\hw(i)} \ \ \forall i\in [1;2n]\}.
$$
An element $w\in W$ is written as a word $[w(1)\,w(2)\,\dots
w(2n)]$. We set 
$$
z=[1 \, 3 \, \dots \, (2n-1) \, (2n) \, \dots 4 \, 2 ] \quad \text{and} \quad y=z^{-1}=[1\,\bar 1\,2\,\bar 2\dots].
$$
With this choice of $z$, if $\ha= \widehat \varepsilon_i - \widehat \varepsilon_j $ for some $1 \leq i < j \leq n$, then $\alpha_\pgg= \varepsilon_i - \varepsilon_j$ and $\alpha_\ngg= \varepsilon_{\bar j } - \varepsilon_{\bar i}.$ If $\ha = \widehat \varepsilon_i + \widehat \varepsilon_j$, then $\alpha_\pgg= \varepsilon_i - \varepsilon_{\bar j}$ and $\alpha_\ngg=\varepsilon_j - \varepsilon_{\bar i}.$
The Lie algebra $\lg$ of $G$ consists of square matrices of size $2n$ with zero trace. We can easily compute that 

\begin{equation*}
    \hlg= \biggl\{ \bmat A & B \\
 C & D
 \emat \, \, : \, \, JAJ=-D^T, \, \, JCJ=C^T, \, \, J B J=B^T, \, \, \Tr(A)+ \Tr(D)=0 \biggr\}.
\end{equation*}
In the previous formula, $A,B,C,D$ are square matrices of size $n$. Note that, the matrix $JA^TJ$ is obtained by flipping $A$ along the principal atidiagonal. Given $1 \leq i,j \leq 2n$, we denote by $E_{ij}$ the canonical basis element, of the space of square matrices of size $2n$, with a coefficient one in row $i$ and column $j$.

\noindent Consider the root $\hb= \wh \varepsilon_i + \wh \varepsilon_j$, for some $1 \leq i < j \leq n$. Then, a possible $\sl_2$ triple for $\hb$, in $\hg$, consists of 

$$X_{\hb}= E_{i\bar j} + E_{j \bar i} \quad H_{\hb}= E_{i i}+ E_{jj} - E_{\bar j \, \bar j} - E_{\bar i \, \bar i} \quad X_{-\hb}= E_{\bar i j} + E_{\bar j i}.$$
The corresponding compatible $\sl_2$ triple, in $\lg$, are

$$
\begin{array}{ccc}
X_{\beta_\pgg}= E_{i \bar j} & H_{\beta_\pgg}=E_{ii}- E_{\bar j \,  \bar j} & X_{-\beta_\pgg}= E_{\bar j i}\\[0.5 em]
X_{\beta_\ngg}= E_{j  \bar i} & H_{\beta_\ngg}=E_{jj}- E_{\bar i \, \bar i} & X_{-\beta_\ngg}= E_{\bar i j}.
\end{array}
$$
Similarly, let $\hb= \wh \varepsilon_i - \wh \varepsilon_j$. Then, a possible $\sl_2$ triple for $\hb$, in $\hg$, consists of 

$$X_{\hb}= E_{i j} - E_{\bar j \, \bar i} \quad H_{\hb}= E_{i i}+ E_{\bar j \, \bar j} - E_{ j j} - E_{\bar i \, \bar i} \quad X_{-\hb}= E_{ j i} - E_{\bar i \, \bar j}.$$
The corresponding compatible $\sl_2$ triple, in $\lg$, are (note the presence of some signs with respect to the previous case)

$$
\begin{array}{ccc}
X_{\beta_\pgg}= E_{i  j} & H_{\beta_\pgg}=E_{ii}- E_{ j  j} & X_{-\beta_\pgg}= E_{ j i}\\[0.5 em]
X_{\beta_\ngg}= - E_{ \bar j  \, \bar i} & H_{\beta_\ngg}=E_{\bar j \, \bar j}- E_{\bar i \, \bar i} & X_{-\beta_\ngg}= - E_{\bar i \, \bar  j}.
\end{array}
$$

\bigskip

Now, fix $n=2$. We chose $\sl_2$ triples for elements of $\hDel^2$, in $\hg$, as above. Moreover, for $\ha_2= 2 \wh \varepsilon_2$, we set 

$$ X_{\ha_2}= E_{23} \quad H_{\ha_2}= E_{22}-E_{33} \quad X_{-\ha_2}= E_{32}.$$
Then, for $t \in \CC^*$ we have that 

$$
x_{\alpha_1}(t)= \bmat 1 & t & 0 & 0\\
0 & 1 & 0 & 0 \\
0 & 0 & 1 & 0\\
0 & 0 & 0 & 1
\emat
\quad
x_{\alpha_2}(t)= \bmat 1 & 0 & 0 & 0\\
0 & 1 & t & 0 \\
0 & 0 & 1 & 0\\
0 & 0 & 0 & 1
\emat
\quad
x_{\alpha_3}(t)= \bmat 1 & 0 & 0 & 0\\
0 & 1 & 0 & 0 \\
0 & 0 & 1 & -t\\
0 & 0 & 0 & 1
\emat,$$
\end{example}
\noindent where, with usual notation, we denote $\alpha_i:= \varepsilon_i - \varepsilon_{i+1}.$

\section{Generalised spherical minors}

\noindent
We introduce some notation. We denote by $w_0 \in W$ the longest element of the Weyl group $W$. For $\lambda \in X(T)$, we set $\lambda^*:= - w_0 \lambda$. If $V$ is a representation of $G$, then the dual space $V^*$ of $V$ is considered as a $G$-module as follows:
$$
(g \cdot \varphi)(v):= \varphi( g^{-1} v) \qquad g \in G, \,\, \varphi \in V^*, \, \, v \in V.
$$
With this notation,  the representation $V(\lambda)^*$ is isomorphic to $V(\lambda^*)$ for any dominant character $\lambda \in X(T)^+$.

\bigskip

\noindent Recall the definition of the monoid of spherical weights $\Lambda_{G/\hG}^+$ given in the introduction. Note that, since $\hG$ is reductive, then 

\begin{equation}
\Lambda_{G/\hG}^+=  \{ \lambda \in X(T)^+ \, : \, V(\lambda^*)^\hG \neq 0 \}
\end{equation}
Moreover, as a consequence of \cite{francone2023multiplicity}[Theorem 1], we have that 

 \begin{equation}
     \label{eq: expression for spherical weights}
 \Lambda_{G/\hG}^+= \{\lambda \in X(T)^+ \, : \, \rho( y \lambda)=0 \}.
 \end{equation}
 Formula \eqref{eq: expression for spherical weights} gives an explicit way to compute the monoid $\Lambda_{G/\hG}^+$.
 Recall that $\mathcal{B}$ denotes the unique set of free generators of the monoid $\Lambda_{G/\hG}^+$.
 Again Formula \eqref{eq: expression for spherical weights}  allows to deduce that, for the fundamental pairs of Theorem \ref{theo:list sph rank min}, the set $\mathcal{B}$ is given as follows. 

\begin{equation}
    \label{eq:sph base of weights}
\begin{array}{l c l }
       1. \quad G= \hG & & \Bc=\emptyset. \\[0.5em]
    2. \quad  G= \hG \times \hG & & \Bc= \{ (\varpi_\ha, \piha^*) \, : \, \ha \in \hDel\}. \\[0.5em]
   3. \quad  (\SL_{2n}, \Sp_{2n}) & & \Bc= \{ \varpi_{2k} \, : \, 1 \leq k < 2n\}.\\[0.5em]
   4. \quad (\Spin_{2n} , \Spin_{2n-1}) & & \Bc= \{ \varpi_1\}.\\[0.5em]
   5. \quad (E_6, F_4) & & \Bc= \{ \varpi_1, \varpi_6 \}.\\[0.5em]
   6. \quad (B_3, G_2) & & \Bc= \{\varpi_3\}.
 \end{array}
 \end{equation}

\begin{remark}
    In \cite{bravi2023multiplication} the monoid of spherical weights is explicitly described for more general spherical subgroups than the ones considered here. The expressions given in \eqref{eq:sph base of weights} can also be found in \cite{bravi2023multiplication}.
\end{remark}

For any spherical weight $\lambda$, the space $V(\lambda)^\hG$ is one dimensional, $\hG$ being spherical in $G$. Let $v_\lambda$ be a generator of $V(\lambda)^{\hG}$. As a consequence of the well known Peter-Weyl theorem, we have an isomorphism of $G$-modules 

\begin{equation}
    \label{eq:spherical peter weyl}
    \begin{array}{r c l}
     \bigoplus_{\lambda \in \Lambda_{G/\hG}^+} V(\lambda)^* & \longto & \CC[G/\hG] \\
      (\psi_\lambda)_\lambda & \longmapsto & \bigl(g \longmapsto \sum_{\lambda} \psi_\lambda(g v_\lambda) \bigr).
    \end{array}
\end{equation}

\noindent For any $\lambda \in \Lambda_{G/\hG}^+$, let $\varphi_\lambda \in V(\lambda)^*$ be the unique $(y U y\inv)$-invariant vector of $V(\lambda)^*$ satisfying $\varphi_\lambda(v_\lambda)=1$. The fact that \eqref{eq:spherical peter weyl} is an isomorphism and $y B y\inv \hG$ is open in $G$ guarantee the existence and unicity of $\varphi_\lambda$. Note that $\varphi_\lambda$ is a weight vector. Its weight is $y \lambda^*=-yw_0\lambda.$

\bigskip

Following \cite{fomin1999double} (see \cite{fomin1999double}[Section 1.4, Eq. (1.8)]), for $\alpha \in \Delta$ we define
\begin{equation}
    \label{s bar}
\sba:= x_\alpha(-1)x_{-\alpha}(1)x_\alpha(-1).
   \end{equation}

\noindent Since the family $\{ \sba : \alpha \in \Delta \}$ satisfies the braid relations, if $\ii=(i_l, \dots , i_i)$ is a reduced expression for some $w \in W$, then the element 
$$
\overline{w}:= \overline{s}_{i_l} \dots \overline{s}_{i_1}
$$
is well defined.

\begin{definition}[Generalised spherical minors]
\label{def:sph minors} Let $w \in W$ and $\lambda \in \Lambda^+_{G/\hG}$. The \textit{generalised spherical minor} $\Delta^\lambda_w \in \CC[G]^\hG$ is the function defined by $$\Delta^\lambda_w(g)= \varphi_\lambda\bigl((\widetilde{w})^{-1} g v_\lambda\bigr), \qquad  (g \in G)$$
where
\begin{equation}
    \label{eq: w tilde spherical}
    \widetilde w= \overline{y}\cdot   \overline{ y \inv  w y} \cdot \overline{ y}^{-1}. 
\end{equation}
\end{definition}

 The following characterisation of the generalised spherical minors is obvious.

\begin{lemma}
    \label{lem:characterisation generalised spherical minors} For any $w \in W$ and $\lambda \in \Lambda^+_{G/\hG}$, the function $\delamw$ is the only element of $\CC[G]$ satisfying the following properties.
    \begin{enumerate}
        \item $\delamw$ is $\hG$ right invariant.
        \item $\delamw$ is $(wyUy \inv w\inv)$ left invariant.
        \item For any $t \in T$ and $g \in G$, we have $\delamw(tg)=(-wy\lambda^*)(t) \delamw(g)$.
        \item $\delamw( \widetilde w)=1$.
    \end{enumerate}
\end{lemma}

 The generalised spherical minors are naturally well behaved with respect to the set of simple roots $y \Delta$. To properly formalise this fact, we introduce some notation. For  $\alpha \in \Delta$, we consider the $\sl_2$-triple $(Z_{y \alpha}, H_{y \alpha}, Z_{- y \alpha})$ for the simple root $y \alpha$  defined by 

\begin{equation}
    \label{eq:sl2 z inverse alpha}
    Z_{y \alpha}:= \overline{y} \cdot X_\alpha, \quad  H_{y \alpha}:=\overline{y}  \cdot H_\alpha, \quad  Z_{-y \alpha}:= \overline{y} \cdot X_{-\alpha}.
\end{equation}

\noindent Then, we denote 

\begin{equation}
\label{eq:z y alpha}
    z_{\pm y  \alpha}(t):= \exp( t Z_{\pm y \alpha})=  \overline{y}  x_{\pm \alpha}(t) \overline{y}\inv.
\end{equation}
Note that 
$$ 
\widetilde{s_{y \alpha}} = z_{y \alpha}(-1)z_{-y \alpha}(1)z_{y \alpha}(-1).
$$
Moreover, if $v,w \in W$ and $\ell(y \inv vwy)=\ell(y \inv v y) + \ell(y \inv w y) $, then $\widetilde{vw}= \widetilde v \widetilde w$ and therefore 
$$
\Delta^\lambda_{vw}= \widetilde v \cdot \Delta^\lambda_w. 
$$
Note also that $\ell(y \inv v y )$ is the length of the Weyl group element $v$ with respect to the set of simple reflections $s_{\delta}$, with $\delta \in y \Delta$. 

If $\delta \in y \Delta$ and $t \in \CC^*$, then the following identities hold. See for example \cite{fomin1999double}[Eq. (2.13)].
 
\begin{equation}
    \label{z alpha s alpha = ..}
    \begin{array}{l}
        z_{\delta}(t) \cdot  \widetilde{s_{\delta}}  = z_{- \delta}(t\inv) \cdot   \delta^\vee(t) \cdot z_{\delta}(-t\inv)\\[0.5em]
        ( \widetilde{s_{\delta}})^{-1} \cdot  z_{-\delta}(t)= z_{-\delta}(-t\inv) \cdot \delta^\vee(t) \cdot z_{\delta}(t\inv).
    \end{array}
\end{equation}

\begin{example}
    \label{ex: sp sl pt 5} We continue Example \ref{ex: sp sl pt 1}. For a positive integer $k$ and a set $S$, we denote by $\mathcal{P}_k\bigl(S \bigr)$ the set of subsets of $S$ of cardinality $k$. Fix an integer $k$ satisfying $1 \leq k < n.$  Then 
    $$V(\varpi_{2k})= \bigwedge^{2k} V. $$
    If $I \in \mathcal{P}_k\bigl([n]\bigr)$, we use the convention that 
    $$
    I= \{i_1 < i_2 < \dots < i_k \}, \quad \overline{I}= \{  \overline{i_k} < \dots < \overline{i_2} < \overline{i_1}\} \quad \text{and} \quad e_I=e_{i_1}  \wedge \cdots \wedge e_{i_k}.
    $$ 
    Then
    \begin{equation}
    \label{eq: ex sp sl 1}
        V(\varpi_{2k})^ \hG= \CC  \biggl( \sum_{I \in \mathcal{P}_k\bigl([n]\bigr)}e_I \wedge e_{\overline{I}} \biggr),
    \end{equation}
    where $[n]$ denotes the set $\{1,2, \dots, n\}.$
     Given two subsets $R, C$ of the set $[2n]$ of the same cardinality, we denote by $D_{R,C} \in \CC[G]$ the determinant of the submatrix whose rows (resp. columns) are indexed by $R$ (resp. $C$). 
     If 
     $$R_k= \{ n-k+1, n-k+2, \dots, n , \overline{n}, \dots , \overline{n-k+2}, \overline{n-k+1} \},$$
     using \eqref{eq: ex sp sl 1}, we easily compute that 
\begin{equation}
    \label{eq: delta pi2k sp sl}
    \Delta^{\varpi_{2k}}_e= \sum_{I \in \mathcal{P}_k\bigl([n]\bigr)} D_{R_k, I \cup \overline{I}}.
\end{equation}
Hence, the generalised spherical minor $\Delta^{\varpi_{2k}}_e$ is the sum of the determinants of the submatrices of size $2k$ whose rows are the $2k$ central ones and whose columns are invariant under the application $i \longmapsto \bar i.$ Similarly, if $w \in W=S_{2n}$, then  
\begin{equation}
    \label{eq: delta pi2k w sp sl}
    \Delta^{\varpi_{2k}}_w=  \pm \sum_{I \in \mathcal{P}_k\bigl([n]\bigr)} D_{ w(R_k), I \cup \overline{I}}.
\end{equation}
The sign in \eqref{eq: delta pi2k w sp sl} depends on $\widetilde w$.

Assume now that $n=2$. 
Moreover, let $k=1$ and 
$$g= \bmat x_{11} & x_{12} & x_{13} & x_{14} \\
x_{21} & x_{22} & x_{23} & x_{24} \\
x_{31} & x_{32} & x_{33} & x_{34} \\
x_{41} & x_{42} & x_{43} & x_{44}
\emat \in G.$$
Then the following expression holds.
$$
\begin{array}{rl}
   \Delta^{\varpi_2}_e(g) = & D_{\{2,3\},\{1,4\}}(g) + D_{\{2,3\},\{2,3\}}(g)\\[0.5em]
   = & (x_{21}x_{34}- x_{24}x_{31}) + (x_{22}x_{33} - x_{23}x_{32}).
\end{array}
$$
In this case, we can compute that
    $$
\overline{s_{\alpha_1}}= \bmat 0 & -1 & 0 & 0\\
1 & 0 & 0 & 0 \\
0 & 0 & 1 & 0\\
0 & 0 & 0 & 1
\emat
\quad
\overline{s_{\alpha_2}}= \bmat 1 & 0 & 0 & 0\\
0 & 0 & -1 & 0 \\
0 & 1 & 0 & 0\\
0 & 0 & 0 & 1
\emat
\quad
\overline{s_{\alpha_3}}= \bmat 1 & 0 & 0 & 0\\
0 & 1 & 0 & 0 \\
0 & 0 & 0 & 1\\
0 & 0 & -1 & 0
\emat.
$$
Moreover, since $y=s_{\alpha_3}s_{\alpha_2}$ and $\ell(y)=2$, we deduce that 
$$\overline{y}= \bmat
1 & 0 &0 &0\\
0& 0& -1& 0\\
0& 0& 0& 1\\
0& -1& 0& 0
\emat
\quad \overline{y}^{-1}=  \bmat
1 & 0 &0 &0\\
0& 0& 0& -1\\
0& -1& 0& 0\\
0& 0& 1& 0
\emat . $$  
Now, for $i \in \{1,2,3\}$, let $\delta_i= y \alpha_i$. By direct computation, we have that 
   $$
\widetilde{s_{\delta_1}}= \bmat 0 & 0 & 0 & 1\\
0 & 1 & 0 & 0 \\
0 & 0 & 1 & 0\\
-1 & 0 & 0 & 0
\emat
\quad
\widetilde{s_{\delta_2}}= \bmat 1 & 0 & 0 & 0\\
0 & 0 & 0 & 1 \\
0 & 0 & 1 & 0\\
0 & -1 & 0 & 0
\emat
\quad
\widetilde{s_{\delta_3}}= \bmat 1 & 0 & 0 & 0\\
0 & 0 & -1 & 0 \\
0 & 1 & 0 & 0\\
0 & 0 & 0 & 1
\emat.
$$ 
\end{example}

\subsection{Basic properties of generalised spherical minors}
\label{sec:basic pros generalised spherical minors}

Recall that, $G$ being semisimple and simply connected, $\hG$ is too because of Lemma \ref{lem:G simply connected, Gh simply connected}. Then, by \cite{popov1994invariant}[Theorem 3.17] we  have that $\CC[G/\hG]=\CC[G]^\hG$ is a factorial algebra of finite type. 

\begin{lemma}
\label{lem:irred generalised spherical minors}
    Let $\lambda, \mu \in \sphweight$ and $ w \in W$. The following holds.
    \begin{enumerate}
        \item $\Delta^{\lambda + \mu}_w= \delamw \Delta^\mu_w.$
        \item If $ \lambda \in \Bc$, then $\delamw$ is irreducible.
        \item If $\lambda, \mu \in \Bc$ and $\lambda \neq \mu$, the ideals generated by $\delamw$ and $\Delta^\mu_w$, in $\CC[G]^\hG$, are different.
    \end{enumerate}
\end{lemma}

\begin{proof}
   The first statement obviously follows from Lemma \ref{lem:characterisation generalised spherical minors}. For the second assertion, it is sufficient to assume $w=e$. Let $\CC\bigl[\sphweight\bigr]$ be the monoid-algebra of $\sphweight$.
   Because of \eqref{eq:spherical peter weyl} and the first statement of Lemma \ref{lem:irred generalised spherical minors}, the map 
   $ \CC\bigl[G/\hG\bigr]^{yUy\inv} \longto \CC\bigl[\sphweight\bigr] $ defined by $\Delta^\nu_e \longmapsto \nu$ is an isomorphism of algebras.
   Since $\sphweight$ is freely generated by $\Bc$, then $\CC\bigl[\sphweight\bigr]$ is a polynomial ring in the variables corresponding to $\Bc$. 
   In particular, the element $b \in \CC\bigl[\sphweight\bigr]$ is irreducible. Since the group $y U y\inv$ is connected and has no non-trivial character, statement two follows. (See for instance \cite{popov1994invariant}[Theorems 3.1, 3.17] or \cite{francone2023minimal}[Lemma 5.2.4]).
   Finally, in the notation of the third statement, assume by contradiction that $(\delamw)= (\Delta^\mu_w)$.
   Since $\CC[G/\hG]^*= \CC[G]^*= \CC^*$ because of the semisimplicity of $G$, then $\delamw= c \Delta^\mu_w$ for some $c \in \CC^*.$ Because of statement 3 of Lemma \ref{lem:characterisation generalised spherical minors}, we deduce that $w y \lambda^*=w y \mu^*$, which implies that $\lambda=\mu.$
\end{proof}

We define the \textit{open cell} of the spherical space $G/\hG$ as   $$
(G/\hG)_0:= z \inv B z \hG/\hG= y B y\inv \hG/\hG.
$$
This is an open affine subsapace of the affine variety $G/\hG.$
\begin{coro}
    \label{cor:vanishing generalised spherical minors}
    The map assigning to $b \in \Bc$ the zero locus of $\Delta^b_e$ gives a bijection between $\Bc$ and the set of irreducible components of the complement of $(G/\hG)_0$ in $G/\hG$. Any such component is of codimension one in $G/\hG.$
\end{coro}

\begin{proof}
    Clearly, for any $b \in \Bc$, we have that $(G/\hG)_0 \subseteq \{ \Delta^b_e \neq 0 \}. $ 
    Since $G/\hG$ is affine and $\CC[G/\hG]$ is factorial, statements 2 and 3 of Lemma \ref{lem:irred generalised spherical minors} imply that the map described in Corollary \ref{cor:vanishing generalised spherical minors} is well defined and injective. By \cite{spherangmin}[Proposition 2.3], there are exactly $| \Delta | - | \hDel |$ irreducible components in $(G/\hG) \setminus (G/\hG)_0.$ We conclude by noticing that $| \Delta | - | \hDel |= |\Bc|$. As a remark, the open cell $(G/\hG)_0$ is clearly pure of codimension one being affine.   
\end{proof}

\begin{coro}
    \label{cor:nonvanishing generalised spherical minors}
    The simultaneous non-vanishing locus of the $\Delta^b_e$, for $b \in \Bc$ , in $G/\hG$, is the open cell $(G/\hG)_0.$
\end{coro}

\begin{proof}
    It's a consequence of Corollary \ref{cor:vanishing generalised spherical minors}.
\end{proof}

For $\delta \in y\Delta$, we denote by $\iota_\delta : \SL_2 \longto G$ the morphism determined by the $\mathfrak{sl}_2$-triple
$(Z_\delta, H_\delta, Z_{-\delta})$. In other words, we have that

\begin{equation*}
   \iota_\delta \begin{pmatrix} 
   1 & t \\
   0 & 1
    \end{pmatrix}=z_\delta(t)
   \quad  \iota_\delta \begin{pmatrix} 
   t& 0 \\
   0 & t\inv
\end{pmatrix}= \delta^\vee(t) \quad 
   \iota_\delta \begin{pmatrix} 
   1 & 0 \\
   t & 1
    \end{pmatrix}=z_{-\delta}(t).
\end{equation*}
Note that
\begin{equation}
    \label{eq:s tilde}
\widetilde{s_\delta}= \iota_\delta \begin{pmatrix} 
    0 & -1\\
    1 & 0
   \end{pmatrix}.
   \end{equation}
Let $\SL_\delta$ be the image of $\iota_\delta.$

\begin{lemma}
\label{lem:twist generalised spherical minors}
Let $ w \in W$, $b \in \Bc$ and $\delta \in y \Delta$. Moreover, let $g \in G$, $s \in \SL_\delta$ and $t \in \CC.$

\begin{enumerate}
    \item If $\langle wy b^*, \delta^\vee \rangle =0$, then $\Delta^b_w(sg)= \Delta^b_w(g).$
    \item If $w\inv \delta \in y \Phi^+$, then $\Delta^b_w \bigl(z_{\delta}(-t) g \bigr)= \Delta^b_w(g).$
    \item If $w\inv \delta \in y \Phi^-$ and  $\langle wy b^*, \delta^\vee \rangle =-1$, then $\Delta^b_w \bigl(z_{\delta}(-t) g \bigr)= \Delta^b_w(g) + t  \Delta^b_{s_\delta w}(g).$
\end{enumerate}
\end{lemma}

\begin{proof}
    Recall that $\Delta^b_w(sg)= \bigl( \widetilde{w}  \varphi_b \bigr)(sg  v_b)= \bigl( s \inv \widetilde{w}  \varphi_b \bigr)(g  v_b).$ 
    But $\widetilde{w} \varphi_b$ is an extremal weight, in the $G$-representation $V(b)^*$, of weight $wyb^*.$
    In particular, if $\langle wy b^*, \delta^\vee \rangle =0$, then $\CC \widetilde{w} \varphi_b$ is the trivial $\SL_\delta$-representation. Hence, we have that $s\inv \widetilde{w} \varphi_b= \widetilde{w} \varphi_b$, which implies the first statement. 
    
    Next, assume that $s= z_\delta(-t).$ If  $w\inv \delta \in y \Phi^+$, then $s \in w y U y\inv w\inv.$ The second assertion follows from statement 2 of Lemma \ref{lem:characterisation generalised spherical minors}.
    
    Finally, suppose that $w\inv \delta \in y \Phi^-$ and  $\langle wy b^*, \delta^\vee \rangle =-1$. 
    Then, by Lemma \ref{lem:characterisation generalised spherical minors}, $\widetilde{w} \varphi_\lambda$ is a lowest weight vector, of weight $-1$, with respect to $\SL_\delta$.
    Using some some standard arguments on the representation theory of $\SL_2$ (see for instance \cite{francone2023minimal}[Lemma 5.4.2]), we compute that 
\begin{equation*}
    \begin{array}{rl}
       \Delta^b_w(z_{\delta}(-t) g) = & \bigl( z_\delta(t) \widetilde{w}  \varphi_b \bigr)(g  v_b)  \\[0.5em]
        = &  \widetilde{w}  \varphi_b(g  v_b) + t \bigl( Z_\delta \widetilde{w}  \varphi_b\bigr )(g  v_b) \\[0.5em]
        = &  \Delta^b_w(g  v_b) + t \bigl( \bigl(\widetilde{s_\delta}\bigr)^{-1}  \widetilde{w}  \varphi_b\bigr)(g  v_b).
    \end{array}
    \end{equation*}
    But if $w\inv \delta \in y \Phi^-$, then $\ell(y \inv  w y)= 1+ \ell(y \inv s_\delta w y) $. Hence $\widetilde{w}= \widetilde{s_\delta} \widetilde{s_{\delta} w}.$ 
    In particular,
    $$ \bigl( \bigl(\widetilde{s_\delta} \bigr)^{-1}  \widetilde{w}  \varphi_b\bigr)(g  v_b)=   \bigr(  \widetilde{s_{\delta} w} \varphi_b\bigr)(g  v_b)= \Delta^b_{s_\delta w}(g).$$
\end{proof}

The following corollary is a useful and obvious consequence of Lemma \ref{lem:twist generalised spherical minors}.
\begin{coro}
    \label{cor: Z delta sph minor.}
Let $ w \in W$, $b \in \Bc$, $\delta \in y \Delta$ and $g \in G$.

\begin{enumerate}
    \item If $\langle wy b^*, \delta^\vee \rangle =0$, then $ Z_\delta \cdot \Delta^b_w= 0 .$
    \item If $w\inv \delta \in y \Phi^+$, then $ Z_\delta \cdot \Delta^b_w= 0.$
    \item If $w\inv \delta \in y \Phi^-$ and  $\langle wy b^*, \delta^\vee \rangle =-1$, then $\ Z_\delta \cdot \Delta^b_w=  \Delta^b_{s_\delta w}(g).$
\end{enumerate}    
\end{coro}

\begin{coro}
\label{cor:sph minors only depend on weights}
    Let $\lambda \in \sphweight$ and $v, w \in W$. If $wy \lambda^*= v y \lambda^*$, then $\Delta^\lambda_w= \Delta^\lambda_v.$
\end{coro}

\begin{proof}
   Let $W_{y\lambda^*}$ be the stabilizer of $y\lambda^*$. It's a classical fact that $W_{y\lambda^*}$ is a parabolic subgroup of $W$ and that 
   $$ W = \bigl\langle s_\delta \, : \, \delta \in y\Delta \quad \text{and} \quad \langle y\lambda^*, \delta^\vee \rangle=0 \bigr\rangle. $$
    We can write uniquely $v=v_1 v_2$ where $v_2 \in W_{y\lambda^*}$ and $v_1$ is the minimal length coset representative of $vW_{y\lambda^*}/W_{y\lambda^*} \in W/W_{y\lambda^*}$.
   Here, the length is computed with respect to the set of simple reflections associated to $y\Delta$. Since $\ell(y\inv v y)= \ell(y\inv v_1 y)+ \ell(y\inv v_2 y)$, then $\Delta^\lambda_v= \widetilde{v_1} \cdot \Delta^\lambda_{v_2}$. 
   Writing similarly $w=w_1 w_2$, we have that $wy \lambda^*= v y \lambda^*$  if and only if $w_1=v_1.$ It follows that it is sufficient to prove the following claim.

\bigskip

 \underline{Claim.} If $v \in W_{y\lambda^*}$, then $\Delta^\lambda_v= \Delta^\lambda_e.$

\bigskip

  \noindent If $\lambda= \sum_{b \in \Bc} n_b b$, for some $n_b \in \NN$, we have that 
   $$ W_{y \lambda^*}= \bigcap_{\{b \, : \, n_b > 0\}} W_{y b^*}.$$
   Then, using the first statement in Lemma \ref{lem:irred generalised spherical minors}, we deduce that it is sufficient to assume that $\lambda=b \in \Bc.$
Suppose that $v= s_\delta u$, for some $\delta \in y \Delta$ such that $\langle yb^*, \delta^\vee \rangle=0$ and $\ell(y\inv v y)=\ell(y\inv u y)+1$. Then, $\Delta^b_v= \widetilde{s_\delta} \Delta^b_u$. Because of statement one of Lemma \ref{lem:twist generalised spherical minors}, we deduce that $\Delta^b_v=\Delta^b_u$. The proof of the claim then follows by induction on $\ell(y\inv v y).$

\end{proof}

Let $K \subseteq T$ be the subgroup generated by the elements $\alpha^\vee(-1)$, for $\alpha \in \Phi$. Note that $K$ is a finite group, and it is normalised by $N_G(T)$. Indeed, if $\alpha \in \Phi$ and $g \in N_G(T)$ represents the class of $w \in W=N_G(T)/T$, then $g \alpha^\vee(-1) g\inv= (w\alpha)^\vee(-1)$.  In particular, $gK=Kg.$ 
The following lemma should be well known to specialists, we include a proof for completeness.
\begin{lemma}
    \label{lem: v bar w bar}
    Let $v,w \in W$, then $\overline{v} \cdot \overline{w} \in \overline{vw} K = K \overline{vw}  $. In particular, $(\overline{v})\inv \in \overline{ v\inv } K = K \overline{ v\inv }. $
\end{lemma}

\begin{proof}
Clearly, the second assertion follows from the first one by considering $w =v \inv$. 
In $\SL_2$ we have that 
    $$\bmat
    0 & -1 \\
    1 & 0 \emat^2= 
    \bmat 
    -1 & 0 \\
   0 & -1 \emat. $$
   Hence, for any $ \alpha \in \Delta$, we have $ \overline{s_\alpha}^2= \alpha^\vee(-1)$.
   Let $\ii $ be a reduced expression of $v$ and $\jj$ be a reduced expression of $w$.
   Let's consider two preliminary cases.

   \bigskip 
   
   \underline{Case 1}. The concatenation of $\ii$ and $\jj$ is a reduced expression for $vw.$ Then, we have that $\overline{v} \cdot \overline{w}= \overline{vw}$ since $\ell(vw)=\ell(v)+\ell(w).$

   \underline{Case 2}. Suppose that $v=v_1 s_\alpha$ and $w=s_\alpha w_1$, for some $\alpha \in \Delta$ and $v_1,w_1 \in W$. Assume furthermore that $\ell(v)=\ell(v_1)+1$, $\ell(w)=\ell(w_1)+1$. Then 
   $$\begin{array}{rl}
      \overline{v} \cdot \overline{w}= & \overline{v_1} \cdot \overline{s_a}^2 \cdot  \overline{w_1} \\[0.5em]
%      = & \overline{v_1} \cdot \overline{w_1} \cdot \bigl( (w_1)\inv \alpha \bigr)^\vee(-1)\\
      = & \bigl( v_1 \alpha \bigr)^\vee(-1) \cdot \overline{v_1} \cdot \overline{w_1}.
   \end{array}$$

 In general, it's a standard fact that %\marginnote{\color{red} add reference?} 
 a reduced expression for $vw$ can be obtained from the concatenation of $\ii$ and $\jj$ by applying a finite number of the following two types of operations.

  \begin{itemize}
      \item[A.] Apply a braid move. That is, replace $s_\alpha s_\beta \cdots$ with $s_\beta s_\alpha \cdots$ (the number of factors in the previous products is a positive integer defining the relations of the Coxeter group $W$), for some simple roots $\alpha \neq \beta.$
      \item[B. ] Remove $s_\alpha s_\alpha$, for some $\alpha \in \Delta.$ 
      %\marginnote{\color{red} explain better?}
  \end{itemize}

  Then, since the elements $\overline{s_\alpha}$, for $\alpha \in \Delta$, satisfy the braid relations, the proof of the general case follows by induction on $ \displaystyle \frac{\ell(v)+ \ell(w) - \ell(vw)}{2}$, using the two previously considered cases.
\end{proof}

The following lemma relies on  having chosen $\sl_2$-triples in $\lg$, compatibly with the ones of $\hlg$. (See Definition \ref{def:compatible sl2 triples} and Formulas \eqref{eq:ops sph rk min type 1},\eqref{eq:ops sph rk min type 2}). 

\begin{lemma}
    \label{lem: w hat bar in G hat}
    If $\hw \in \hW$, then $\overline{\hw} \in \hG$.
\end{lemma}

\begin{proof}
Recall that, for any $\alpha \in \Delta$, then $\overline{s_\alpha}=x_\alpha(-1) x_{-\alpha}(1) x_\alpha(-1)$. 
From the previous expression and \eqref{eq:ops sph rk min type 1}, we deduce that if $\alpha \in \Delta^1$, then $\overline{s_\alpha} \in \hG.$ If $\ha \in \hDel^2$ and $\bar \rho \inv (\ha)= \{\alpha_\pgg, \alpha_\ngg\}$, using the strong orthogonality of $\alpha_\ngg$ and $\alpha_\pgg$ and Lemma \ref{lem: reflections sph rk min}, we compute that 

\begin{equation*}
    \begin{array}{rl}
     \overline{s_\ha}= \overline{s_{\alpha_\ngg} s_{\alpha_\pgg}}=   &  x_{\alpha_\ngg}(-1) x_{-\alpha_\ngg}(1) x_{\alpha_\ngg}(-1)x_{\alpha_\pgg}(-1) x_{-\alpha_\pgg}(1) x_{\alpha_\pgg}(-1)  \\
        = & x_{\alpha_\ngg}(-1) x_{\alpha_\pgg}(-1) x_{-\alpha_\ngg}(1)x_{-\alpha_\pgg}(1) x_{\alpha_\ngg}(-1)   x_{\alpha_\pgg}(-1) \\
        = & x_\ha(-1) x_{-\ha}(1) x_\ha(-1).
    \end{array}
\end{equation*}
The last equality follows from \eqref{eq:ops sph rk min type 2} and immediately implies that $\overline{s_\ha} \in \hG$. The statement then follows from Lemma \ref{lem: compatible red expressions spherical}.
\end{proof}

\begin{coro}
\label{cor: hw tilde almost in hG}
    If $\hw \in \hW$, then $\widetilde{\hw} \in \hG K \cap K \hG$.
\end{coro}

\begin{proof}
    Obvious from Lemmas \ref{lem: v bar w bar}, \ref{lem: w hat bar in G hat} and the definition of $\widetilde{\hw}.$
\end{proof}

\begin{coro}
\label{cor: sph minors on e}
   Let $\lambda \in \sphweight$ and $ w \in W$. The following holds.
   \begin{enumerate}
       \item If $\rho (w y \lambda^*) \neq 0$, then $\Delta^\lambda_w(e)=0.$
       \item If $\hw \in \hW$, then $\Delta^\lambda_{\hw}(e) \in \{\pm 1\}.$
   \end{enumerate}
\end{coro}

\begin{proof}
    If $\rho (w y \lambda^*) \neq 0$, let $\widehat t \in \hT$ such that $w y \lambda^*(\widehat t) \neq 1$. 
    Using statements one and three of Lemma \ref{lem:characterisation generalised spherical minors}, we deduce that 
    $$\bigl(-w y \lambda^*\bigr) (\widehat t)  \Delta^\lambda_w(e)=  \Delta^\lambda_w( \widehat t )= \Delta^\lambda_w(e).$$
    Hence, $\Delta^\lambda_w(e)=0.$ In the second case, using Corollary \ref{cor: hw tilde almost in hG}, we can write $\widetilde{\hw}=  \hg k$ for some $\widehat g \in \hG$ and $k \in K$. Then, again by Lemma \ref{lem:characterisation generalised spherical minors}, we compute that 
    $$\Delta^\lambda_{\hw}(e)= \Delta^\lambda_e(k\inv \widehat g \inv)= (-y \lambda^*)(k \inv).$$
    But for any $\nu \in X(T)$, $\nu(K) \subseteq \{\pm 1\}$ because of the definition of $K.$
\end{proof}

\begin{remark}
    In the notation of Corollary \ref{cor: sph minors on e}, it may happen that $\Delta^\lambda_{\hw}(e)=-1$. An example is given in Example \ref{ex:relations minors sl sp n=2}.
\end{remark}

\subsubsection{On the choice of \texorpdfstring{$z$}{z}}

In the previous section, we made a choice of an element $z \in Z_m$. This choice affects the elements $\widetilde w$, for $w \in W$, on which we rely to define the generalised spherical minors. We prove that the generalised spherical minors actually do not depend, up to a hard to predict sign, on the choice of $z$.

\bigskip

Recall that $y=z\inv$. Let $z' \in Z_m$ and $y':= (z')\inv$. For $w \in W$, let $\widetilde{w}'= \overline{ y'} \cdot \overline{ (y')\inv w y' } \cdot \bigl(\overline{y'}\bigr)^{-1}.$ 
Let $\lambda \in \sphweight$. We denote by $D^\lambda_w$ the generalised spherical minor, associated to $\lambda$ and $w$, defined with respect to $z'$. As in the previous section, $\Delta^\lambda_w$ denotes the minor defined with respect to $z$. 

\begin{lemma}
    \label{lem:sph minors independent z}
    Let $w, v \in W$ and $\lambda, \mu \in \sphweight$ such that $wy \lambda^*=vy'\mu^*$. Then $\Delta^\lambda_w= \pm D^\mu_v$.
\end{lemma}

\begin{proof}
    Clearly, we have that $\lambda=\mu.$ Both $\Delta^\lambda_w$ and $D^\lambda_v$ are the image, via the map \eqref{eq:spherical peter weyl}, of a weight vector of $V(\lambda^*)$ of weight $wy\lambda^*.$ Being $wy\lambda^*$ an extremal weight of $V(\lambda^*)$, it follows that $\Delta^\lambda_w= c D^\lambda_v$ for some $c \in \CC^*$. Using Lemma \ref{lem:characterisation generalised spherical minors}, we deduce that $c= \Delta^\lambda_w( \widetilde{v}').$
By Lemma \ref{lem: v bar w bar}, we deduce that there exist $k \in K$ such that
    $$ 
    \begin{array}{rl}
        \Delta^\lambda_w( \widetilde{v}')=  & \Delta^\lambda_e \bigl( (\widetilde{w})\inv \widetilde{v}' \bigr)  \\[0.5 em]
        =  & \Delta^\lambda_e\biggl( \biggl(\widetilde{v \inv w}\biggr)^{-1} k\biggr)\\[0.8em]
        = & \Delta_{v\inv w}(k)\\[0.5em]
        = & (-v\inv w y \lambda^*)(k) \Delta^\lambda_{v\inv w}(e).
    \end{array}$$
    Then, it is sufficient to prove that $\Delta^\lambda_{v\inv w}(e) \in \{\pm 1 \}.$
    By \cite{francone2023multiplicity}[Lemma 3], there exists a $\hw \in \hW$ such that $y'= \hw y.$ In particular $$v\inv w y \lambda^*= \hw y \lambda^*.$$
    Using Corollaries \ref{cor:sph minors only depend on weights} and \ref{cor: sph minors on e}, we deduce that
$$
\Delta^\lambda_{v\inv w}(e)= \Delta^\lambda_{\hw}(e) \in \{\pm 1\}.
$$
%\marginnote{\color{red} example for $G \subset G \times G$ in which the sign is -?}
\end{proof}

\section{Proof and description of Theorem \ref{thm:fundamental identities sph mnors}}
\label{sec:fund identities}

%\marginnote{\color{red} l'elemento z va definitio nelle varie sezioni}
 
Sections \ref{sec:diag} to \ref{sec:G2} correspond to the spherical pairs of minimal rank of Theorem \ref{theo:list sph rank min}.
Each section contains a detailed description of the element $z$ with respect to which the generalised spherical minors appearing in Theorem \ref{thm:fundamental identities sph mnors} are defined, along with the description of the matrix $A$, the elements $\hw_b$ and the sets $U_b$. 
In each section we complete the proof of Theorem \ref{thm:fundamental identities sph mnors} for the corresponding pairs. 
The proof of Theorem \ref{thm:fundamental identities sph mnors} begins with a general reduction argument.
In Sections \ref{sec:diag} to \ref{sec:G2}, we extensively use the notation of \cite{Bourb} for the root systems.
%For the pair $(B_3, G_2)$, we chose $z= s_3s_2s_3$. 
 
%Then, we complete the proof by a case by case analysis, preceded by a 

\subsection{Some reductions}
Let $(G,\hG)$ be one of the pairs of Theorem \ref{theo:list sph rank min} and $b \in \Bc$. The notation $z,A, \hw_b$ and $U_b$ stand for the elements defined in the corresponding section below.
Consider the spherical weight
\begin{equation}
    \label{eq:lambda b sph minors}
    \lambda_b:= \biggl( \sum_{b' \in \Bc \setminus \{b\}} -a_{b,b'}b' \biggr) \in \sphweight.
\end{equation}
The identity \eqref{eq:fundamental iddentity sh minors} can be rephrased as

\begin{equation}
\label{eq:fund identity sph minors versione branching}
 \sum_{u \in U_b} (-1)^{\ell(y\inv u y)} \Delta^b_{wu} \Delta^b_{w u \hw_b} = \epsilon_b \Delta^{\lambda_b}_w.
\end{equation}

 Recall that if $\lambda$ is a spherical weight, and $v, v' \in W$ satisfy $\ell(y\inv v v' y)= \ell(y\inv v y) + \ell(y\inv v' y)$, then $\Delta^\lambda_{vv'}= \widetilde{v} \cdot \Delta^{\lambda}_{v'}$.
 Hence, for proving Theorem \ref{thm:fundamental identities sph mnors} it is sufficient to assume that $w=e$.
 We set
\begin{equation}
    \label{eq:fb spheric minors}
    f_b =   \sum_{u \in U_b} (-1)^{\ell(y\inv u y)} \Delta^b_{u} \Delta^b_{ u \hw_b} \in \CC[G]^\hG.
\end{equation}
  Using Lemma \ref{lem:characterisation generalised spherical minors}, we deduce that  the following 3 claims imply \eqref{eq:fund identity sph minors versione branching}, from which Theorem \ref{thm:fundamental identities sph mnors} follows.

\bigskip

    \noindent \underline{Claim 1:} For any $\delta \in y \Delta$, we have that  $Z_\delta \cdot f_b=0.$

\bigskip

    \noindent \underline{Claim 2':} For any $u \in U_b$, we have that $ u y b^* + u \hw_b y b^* = y \lambda_b^*.$
    
    \noindent Moreover, in all cases we verify that $U_b \subseteq W_{y \lambda_b^*}$. Hence, Claim 2' is equivalent to Claim 2 below.

    \bigskip

    \underline{Claim 2:} $ \quad yb^* + \hw_b y b^*= y \lambda_b^*.$

    \bigskip

    \noindent \underline{Claim 3':} $ \quad f_b(e) \in \{\pm 1\}.$ 
    
    \noindent Moreover, in all cases we verify that  $\hw_b \in \hW$. 
    Therefore, we have that $\Delta^b_e\Delta^b_{\hw_b}(e) \in \{ \pm 1\}$ by Corollary \ref{cor: sph minors on e}. Hence, again by Corollary \ref{cor: sph minors on e}, Claim 3' is implied by Claim 3 below.

    \bigskip

    \underline{Claim 3} For any $u \in U_{b} \setminus \{e\}$, we have that  $\rho(u y b^*) \neq 0$ or $\rho(u \hw_b y b^*) \neq 0.$

\begin{remark}
    Claim 3 implies that $\epsilon_b= \Delta^b_{\hw_b}(e).$
\end{remark}

 In the diagonal case ($G= \hG \times \hG$), we don't need the above reductions.
 Indeed, we prove that generalised spherical minors for the diagonal case coincide with Fomin-Zelevinsky generalised minors. Theorem \ref{thm:fundamental identities sph mnors} is then deduced from \cite{fomin1999double}[Theorem 1.17]. Of course, it is possible to prove Theorem \ref{thm:fundamental identities sph mnors} by proving the above claims, but this would recover the exact same proof of Fomin and Zelevinsky. That's why we prefer to deduce Theorem \ref{thm:fundamental identities sph mnors}, for $G=\hG \times \hG$, from \cite{fomin1999double}[Theorem 1.17].
 Note also that if $G=\hG$, Theorem \ref{thm:fundamental identities sph mnors} is empty. In the other cases of Theorem \ref{theo:list sph rank min}, we prove claims 1 to 3, from which Theorem \ref{thm:fundamental identities sph mnors} follows. 

\begin{remark}
    The proof of Claim 1 heavily exploits the following assertion, which is proved by case by case calculations. It would be interesting to have a conceptual general argument for it.

\bigskip 

    \underline{Assertion:} Let $(G,\hG)$ be as in Theorem \ref{theo:list sph rank min}, $b \in \Bc$ and $\delta \in y \Delta$. For any $v \in U_b \cup U_b \cdot\hw_b$ such that $v \inv \delta \in y \Phi^-$, we have  $\langle v y b^* \, , \, \delta^\vee \rangle = -1$ and $s_\delta v \in  U_b \cup U_b \cdot \hw_b $.
\end{remark}

\begin{remark}
Recall the decomposition of the ring $\CC[G]^\hG$ into $G$-modules given by \eqref{eq:spherical peter weyl}. 
Let $\lambda, \mu \in \sphweight$.
We can define $V(\lambda) \cdot V(\mu)$ as the $G$-submodule of $\CC[G]^\hG$ generated by the products $fg$ for $f \in V(\lambda)$ and $g \in V(\mu).$ 
The problem of  understanding the decomposition of $V(\lambda) \cdot V(\mu)$, for the various $\lambda, \mu \in \sphweight$, into irreducible $G$-modules is studied in detail in \cite{bravi2023multiplication}.    Because of \eqref{eq:fund identity sph minors versione branching}, Theorem \ref{thm:fundamental identities sph mnors} implies that $V(\lambda_b^*)$ is a component of $V(b^*) \cdot V(b^*).$ 
This obviously implies that $V(\lambda_b^*)$ appears in the $G$-module decomposition of $V(b^*) \otimes V(b^*).$
\end{remark}
 \subsection{The diagonal case}
\label{sec:diag}
 We consider a simple, simply connected group $\hG$ diagonally embedded in $G=\hG \times \hG.$ Let $B= \hB \times \hB$ and $T=\hT \times \hT$. We have that $W=\hW \times \hW$
  and we set
  $$
  z:= (\hw_0, e)= y \in Z_m.
  $$
  Therefore, 
  $$
  y \Delta=\bigl \{ (-\ha, 0), (0, \ha) \, : \, \ha \in \hDel \bigr \} \quad \text{and} \quad \Bc= \bigl \{ (\varpi_{\ha}, \varpi_{\ha}^*) \, : \, \ha \in \hDel \bigr \}.
  $$

\bigskip

For $\ha \in \hDel$, let $b_{\ha}= (\piha, \piha^*).$ 
The Cartan matrix $A$ of the small root system coincides with the Cartan matrix of $\hG$. That is 
$$
a_{b_{\ha}, b_{\hb}}:= \langle \hb , \ha^\vee \rangle.
$$
We set 

$$
\hw_{b_{\ha}}:= (s_\ha, s_\ha) \quad \text{and} \quad U_{b_\ha}:= \bigl\{ e , (s_\ha, e) \bigr \}.
$$

  \begin{lemma}
      \label{lem:tilde equals bar tensor}
      For any $v \in W$, we have $\widetilde{v}= \overline{v}$.
  \end{lemma}

\begin{proof}
    Because of the choice of $z$, then  $\ell(y \inv v y)= \ell(v).$ Indeed, $\ell(y \inv v y)$ equals the length of $v$ with respect to the set of simple reflections, of the Coxeter group $W$, associated to $y \Delta.$
    But $\Delta$ and $y \Delta$ induce the same set of simple reflections.
    Hence, it is sufficient to prove that, for any $\alpha \in \Delta$, we have $\widetilde{s_\alpha}= \overline{s_\alpha}.$
    With a little abuse, for $\hw \in \hW$ we denote by $\overline{ \hw} \in \hG$ the representative of $\hw$ computed with respect to the $\sl_2$-triples $(X_{-\ha}, H_{\ha}, X_\ha)$, for $\ha \in \hDel$. 
    
Recall that the $\sl_2$-triples of $\lg$ associated to the elements of $\Delta$ are compatible in the sense of Definition \ref{def:compatible sl2 triples}. It follows that, for any $\widehat v, \hw \in \hW$, we have 
$$\overline{(\widehat v, \hw)}= \bigl( \overline{\widehat v} , \overline{ \hw} \bigr).$$
Hence, if $\alpha=(0, \ha) $ for some $\ha \in \hDel$, we have that
$$\begin{array}{rl}
  \widetilde{s_\alpha} =   &  \bigl( \overline{\hw_0}, e \bigr) \bigl(e, \overline{s_{\ha}}\bigl) \bigr( \bigl(\overline{\hw_0} \bigr)^{-1} ,e \bigr) \\[0.5 em]
   =   & (e, \overline{s_{\ha}})\\[0.5 em]
   = & \overline{s_\alpha}.
\end{array}$$
Now, since $1+ \ell(s_{\hw_0\ha} \hw_0)= \ell(\hw_0)= \ell(\hw_0 s_{\ha}) + 1$, we deduce that 
$$ \overline{s_{\hw_0 \ha}} \cdot \overline{ s_{\hw_0 \ha} \hw_0}= \overline{ \hw_0}= \overline{\hw_0 s_{\ha}} \cdot \overline{s_{\ha}}.$$
In particular, if $\alpha=(- \hw_0 \ha, 0)$ for some $\ha \in \hDel$, then

$$
\begin{array}{r l}
   \widetilde{s_\alpha}=  & \bigl( \overline{s_{\hw_0 \ha}} \cdot \overline{s_{\hw_0 \ha} \hw_0}, e \bigr) \bigl(\overline{s_{\ha}}, e \bigr) \bigl( (\overline{s_{\ha}})^{-1}\cdot (\overline{\hw_0 s_{\ha}})^{-1} , e \bigr)\\[0.5em]
   =& \bigl(\overline{s_{\hw_0 \ha}},e )\\[0,5em]
   = & \overline{s_{\alpha}}.
\end{array}$$
In the second equality, we used that $ \overline{s_{\hw_0 \ha} \hw_0} =\overline{\hw_0 s_{\ha}}.$ This completes the proof.
\end{proof}

Let $\pi : \hG \times \hG \longto \hG$ be the map sending $(\widehat g_1, \hg_2) $ to $\hg_1 \hg_2^{-1}.$
Clearly, the map $\pi$ induces a $G$-equivariant isomorphism between $G/\hG$ and $\hG.$
For $ \widehat v, \hw \in \hW$ and $\ha \in \hDel$, we denote by $\Delta_{\widehat v \piha, \, \hw \piha} \in \CC[\hG]$ the Fomin-Zelevinsky generalised minor, as defined in \cite{fomin1999double}. We worn the reader that the notation for generalised spherical minors is very similar to the one for Fomin-Zelevinsky generalised minors. 
\[
\begin{array}{cc cc}
    \text{FZ generalised minors} & & \text{Generalised spherical minors}  & \\[0.5em]
   \Delta_{\widehat v \piha, \, \hw \piha}  &  & \Delta^{(\piha, \piha^*)}_{(\widehat v, \widehat w)} & (\ha \in \hDel, \,   \widehat v, \hw \in \hW).
\end{array}
\]

\begin{lemma}
    \label{lem:sph minors tensor prod equals FZ minors}
    For any $\widehat v, \hw \in \hW$ and $\ha \in \hDel$, we have that $\pi^* \bigl(\Delta_{\widehat v \piha , \, \hw \piha} \bigr)= \Delta^{(\piha, \piha^*)}_{(\widehat v, \widehat w)}$.
\end{lemma}

\begin{proof}
    Because of Lemma \ref{lem:tilde equals bar tensor} and the fact that $\pi$ is $G$-equivariant, we have that 
    $$
    \pi^* \bigl(\Delta_{\widehat v \piha, \, \hw \piha} \bigr)= \bigl(\overline{\hv}, \overline{\hw}\bigr) \cdot \pi^* \bigl(\Delta_{\piha, \, \piha}\bigr) \quad \text{and} 
    \quad \Delta^{(\piha, \piha^*)}_{(\widehat v, \widehat w)}= \bigl(\overline{\hv}, \overline{\hw}\bigr)\cdot \Delta^{(\piha, \piha^*)}_{(e,e)}.$$
    Hence, we can assume that $\hv=e=\hw$. Note that  $\pi^* \bigl(\Delta_{\piha, \, \piha}\bigr)$ is clearly $\hG$-right invariant.
    Moreover, because of \cite{fomin1999double}[Eq. (1.6),(1.10)], then $\pi^* \bigl(\Delta_{\piha,\, \piha}\bigr)$ is  $y U y\inv $ left-invariant.
    Again by \cite{fomin1999double}[Eq. (1.6),(1.10)], if $\widehat h_1 , \widehat h_2 \in \hT$ and $\hg_1, \hg_2 \in \hG$, then 
    $$
    \begin{array}{rl}
         \pi^* \bigl(\Delta_{\piha, \, \piha}\bigr)\biggl(\bigl( \widehat h_1 \hg_1, \widehat h_2 \hg_2 \bigr) \biggr)=
         & \piha( \widehat h_1) \piha(\widehat h_2 \inv  ) \pi^* \bigl(\Delta_{\piha, \, \piha}\bigr)\biggl((\hg_1, \hg_2)\biggr)\\[1 em]
         = & \biggl( - y \bigl( (\piha, \piha^*)^* \bigr) \biggr) \biggl((\widehat h_1, \widehat h_2) \biggr)\pi^*\bigl(\Delta_{\piha, \, \piha}\bigr)\biggl((\hg_1, \hg_2)\biggr).
    \end{array}
    $$
    Indeed, $-y \bigl( (\piha, \piha^*)^* \bigr)= -(\hw_0,e)(-\hw_0 \piha, \piha)=(\piha, - \piha).$
    Moreover, by \cite{fomin1999double}[Eq. (1.6),(1.10)] we have that 
    $$\pi^*\bigl(\Delta_{\piha, \, \piha}\bigr)\bigl((e,e)\bigr)=1.$$
Then the statement follows by Lemma \ref{lem:characterisation generalised spherical minors}.
\end{proof}

\begin{proof}[Proof of Theorem \ref{thm:fundamental identities sph mnors}, diagonal case.]
Because of Lemma \ref{lem:sph minors tensor prod equals FZ minors}, Theorem \cite{fomin1999double}[Theorem 1.17] implies Theorem \ref{thm:fundamental identities sph mnors}. Indeed, \cite{fomin1999double}[Theorem 1.17] is equivalent to Theorem \ref{thm:fundamental identities sph mnors} and the fact that for any $\ha \in \hDel$ we have $\epsilon_{b_\ha}= 1$.
    
\end{proof}

\subsection{ \texorpdfstring{$(\SL_{2n}, \Sp_{2n})$}{SL2n, Sp2n}}
We consider the element $z \in Z_m$ defined in Example \ref{ex: sp sl pt 1}.
The Cartan matrix $A$ of the small root system of the pair $(\SL_{2n}, \Sp_{2n})$ is of type $A_{n-1}$. In particular, for $1 \leq k,m <n$, we have that 
$$
a_{\varpi_{2k}, \varpi_{2m}}:= \begin{cases}
    2 & \text{if} \quad k=m\\
    -1 & \text{if} \quad k=m\pm 1 \\
    0 & \text{otherwise}.
\end{cases}
$$
For $1 \leq i \leq 2n-1$, we denote 
$$ \delta_i:= y \alpha_i \quad s_i:=s_{\alpha_i} \quad \tau_i:= s_{\delta_i}= y s_i y \inv.$$
Note that the reflections $\tau_i$ are the simple reflections of the Coxeter group $W$, with respect to the base $y \Delta$ of the root system $\Phi.$ 
Fix $1 \leq k < n$. We set 

$$
\begin{array}{ll}
   \hw_{\varpi_{2k}}:= \tau_{2(n-k)} \tau_{2(n-k)+1}\tau_{2(n-k)-1}\tau_{2(n-k)} \\[0.5em]
    U_{\varpi_{2k}}:= \bigl \{ e , \tau_{2(n-k)}, \tau_{2(n-k)+1} \tau_{2(n-k)} \bigr\}. 
\end{array}  $$

A direct computation shows that $\hw_{\varpi_{2k}}= (n-k, n-k+1) (\overline{ n-k+1}, \overline{n-k})$. Hence, we have that $\hw_{\varpi_{2k}} \in \hW.$  
Note that 
$$
\lambda_{\varpi_{2k}}= \varpi_{2(k-1)}+\varpi_{2(k+1)},
$$
with the convention that
$$
\varpi_0= 0= \varpi_{2n}.
$$
Note that $\lambda_{\varpi_{2k}}^*= \varpi_{2(n-k-1)} + \varpi_{2(n-k+1)}$ is stabilised by $s_{2(n-k)-1}, s_{2(n-k)}$ and $s_{2(n-k)+1}$. 
Hence, the set $U_{\varpi_{2k}}$ is contained in the stabiliser of $y \lambda_{\varpi_{2k}}^*$. We are now in position to prove claims 1 to 3. Note that $\varpi_{2k}^*= \varpi_{2(n-k)} $.

\begin{proof}[Proof of Claim 1.] 
The computations needed for the proof, that can be easily carried out by hand, are summarised in Table \ref{tab:sl sp}.
The latter should be read as follows. Let $\delta   \in y \Delta$ and $v \in U_{\varpi_{2k}} \cup U_{\varpi_{2k}} \cdot \hw_{\varpi_{2k}}$. In the entry indexed by the row $\delta$ and the column $v$, we put a symbol $\color{Green}{+}$ if $v \inv \delta \in y \Phi^+$. Conversely, if $v\inv \delta \in y \Phi^-$, the entry is filled with $\color{red}{\langle v y \varpi_{2k}^*, \delta^\vee \rangle}.$ Note that $\langle v y \varpi_{2k}^*, \delta^\vee \rangle= \langle  y \varpi_{2k}^*, (v \inv\delta)^\vee \rangle$.

    \begin{table}[H]
        \centering
        \begin{tabular}{l|c|c|c|c|c|c}
             &  $e$ & $\hw_{\varpi_{2k}}$ & $\tau_{2(n-k)}$ & $\tau_{2(n-k)} \hw_{\varpi_{2k}}$ & $\tau_{2(n-k)+1}\tau_{2(n-k)}$ &  $\tau_{2(n-k)+1}\tau_{2(n-k)} \hw_{\varpi_{2k}}$\\
             \hline 
          $\delta_{2(n-k)}$   & $\color{Green}{+}$ & \color{red}{-1} & \color{red}{-1} & $\color{Green}{+}$ & $\color{Green}{+}$  & $\color{Green}{+}$\\
             \hline 
            $\delta_{2(n-k)+1}$ & $\color{Green}{+}$ & $\color{Green}{+}$ & $\color{Green}{+}$ & \color{red}{-1} &  \color{red}{-1} & $\color{Green}{+}$\\
             \hline 
             $\delta_{2(n-k)-1}$ & $\color{Green}{+}$&  $\color{Green}{+}$& $\color{Green}{+}$ & \color{red}{-1} &  $\color{Green}{+}$  & \color{red}{-1} \\
             \hline 
            \makecell{$\delta_i$  \\   \tiny{$|i - 2(n-k) | \geq 2$}}  & $\color{Green}{+}$ & $\color{Green}{+}$ & $\color{Green}{+}$ & $\color{Green}{+}$ & $\color{Green}{+}$ & $\color{Green}{+}$
        \end{tabular}
        \caption{Infinitesimal action \texorpdfstring{$(\SL_{2n}, \Sp_{2n})$}{SL2n, Sp2n}}
        \label{tab:sl sp}
    \end{table}

We explain, as an example, a possible way of filling the column indexed by $$\tau_{2(n-k)} \hw_{\varpi_2k}= \tau_{2(n-k)+1}\tau_{2(n-k)-1}\tau_{2(n-k)}.$$ 
Recall that an inversion of $w \in W$ is a positive root $\beta \in \Phi^+$ such that $w \cdot \beta \in \Phi^-$.
The set of simple inversions of $s_{2(n-k)}s_{2(n-k)-1}s_{2(n-k)+1}$ is $\{\alpha_{2(n-k)+1}, \alpha_{2(n-k)-1}\}.$ Hence, all the entries of the chosen column are $\color{Green}{+}$, except for the ones whose row index are $\delta_{2(n-k)-1}$ or $\delta_{2(n-k)+1}$. Moreover, 

$$\begin{array}{rl}
    \langle \tau_{2(n-k)} \hw_{\varpi_{2k}}  y \varpi_{2k}^* \,,\, \delta_{2(n-k)\pm 1}^\vee\rangle = & 
   \langle y \varpi_{2k}^* \,, \,y s_{2(n-k)}s_{2(n-k)-1}s_{2(n-k)+1} \alpha_{2(n-k)\pm 1}^\vee \rangle  
   \\[0.5em]
    = & \langle  \varpi_{2k}^* \,, \, s_{2(n-k)}s_{2(n-k)-1}s_{2(n-k)+1} \alpha_{2(n-k)\pm 1}^\vee \rangle \\[0.5em]
    = & \langle  \varpi_{2(n-k)} \,,\, \bigl( -\alpha_{2(n-k)\pm 1} -\alpha_{2(n-k)}\bigr)^\vee \rangle \\[0.5em]
    = & -1.
\end{array} $$

Corollary \ref{cor: Z delta sph minor.} and Table \ref{tab:sl sp} easily allow to conclude. Indeed, if 
$$
\delta \in y \Delta \setminus \{ \delta_{2(n-k)-1}, \delta_{2(n-k)},\delta_{2(n-k)+1} \},
$$
then $Z_{\delta} \cdot f_{\varpi_{2k}}=0.$ While, if $\delta=\delta_{2(n-k)}$, using the Leibniz rule, we compute that  

    $$ Z_\delta \cdot f_{\varpi_{2k}} = \Delta^{\varpi_{2k}}_e \Delta^{\varpi_{2k}}_{\tau_{2(n-k) \hw_{\varpi_{2k}}}} -  \Delta^{\varpi_{2k}}_{\tau_{2(n-k)}\tau_{2(n-k)}}  \Delta^{\varpi_{2k}}_{\tau_{2(n-k) \hw_{\varpi_{2k}}}} =0 .$$
    The case of $\delta_{2(n-k) \pm 1}$ is almost identical to the previous one.
\end{proof}

\begin{proof}[Proof of Claim 2.] The identity we need to prove is equivalent to 
  $$  \varpi_{2(n-k)}+  s_{2(n-k)} s_{2(n-k)+1} s_{2(n-k)-1}s_{2(n-k)} \varpi_{2(n-k)} = \varpi_{2(n-k-1)} + \varpi_{2(n-k+1)},$$
which follows by an immediate calculation.
\end{proof}

\begin{proof}[Proof of Claim 3.]
     We have that
    $$\begin{array}{rl}
         \tau_{2(n-k)} y \varpi_{2k}^*= & y s_{2(n-k)}\varpi_{2(n-k)}  \\
        = & y \varpi_{2(n-k)}- \delta_{2(n-k)}.
    \end{array}$$
    But $\rho\bigl( y \varpi_{2(n-k)}\bigr)=0$ because of \eqref{eq: expression for spherical weights}, while $\rho\bigl(\delta_{2(n-k)}\bigr) = \rho\bigl( \varepsilon_{\overline{n-k}} - \varepsilon_{n-k+1}\bigr) \neq 0.$ Hence,
    $$\rho \bigl(  \tau_{2(n-k)} y \varpi_{2k}^* \bigr) \neq 0.$$
    Similarly, one verifies, by direct computation, that 
    $$\rho\bigl( \tau_{2(n-k)+1} \tau_{2(n-k)} y \varpi_{2k}^* \bigr) \neq 0.$$
\end{proof}

\begin{example}
    \label{ex:relations minors sl sp n=2}
    Let $n=2, k=1$.
    In the notation of Example \ref{ex: sp sl pt 5}, we have that 
$$ \begin{array}{r l }
    \Delta^{\varpi_2}_e(g) 
   = & (x_{21}x_{34}- x_{24}x_{31}) + (x_{22}x_{33} - x_{23}x_{32})\\
   \Delta^{\varpi_2}_{\tau_2 \tau_3 \tau_1 \tau_2}(g)= & - (x_{11}x_{44} - x_{14}x_{41}) - (x_{12}x_{43} - x_{13}x_{42})\\
   \Delta^{\varpi_2}_{\tau_2}(g)= &  (x_{31}x_{44} - x_{34}x_{41}) + (x_{32}x_{43} - x_{33}x_{42})\\
   \Delta^{\varpi_2}_{ \tau_3 \tau_1 \tau_2}(g)= &  (x_{11}x_{24} - x_{14}x_{21}) + (x_{12}x_{23} - x_{13}x_{22})\\
   \Delta^{\varpi_2}_{ \tau_3  \tau_2}(g)= &  - (x_{21}x_{44} - x_{24}x_{41}) - (x_{22}x_{43} - x_{23}x_{42})\\
    \Delta^{\varpi_2}_{ \tau_1  \tau_2}(g)= &  - (x_{11}x_{34} - x_{14}x_{31}) -(x_{12}x_{33} - x_{13}x_{32}).   
\end{array}$$
Expanding $f_{\varpi_2}= \Delta^{\varpi_2}_e \Delta^{\varpi_2}_{\tau_2 \tau_3 \tau_1\tau_2}-  \Delta^{\varpi_2}_{\tau_2}\Delta^{\varpi_2}_{\tau_3 \tau_1 \tau_2} + \Delta^{\varpi_2}_{\tau_3\tau_2}\Delta^{\varpi_2}_{ \tau_1 \tau_2}$ using the previous expressions, we obtain that 
$$f_{\varpi_2}(g)= - \det(g)= - 1.$$
Relation \eqref{eq:fundamental iddentity sh minors}, for $w=e$, predicts that 
    \begin{equation}
        \label{eq:relation slsp n=2}
        \Delta^{\varpi_2}_e \Delta^{\varpi_2}_{\tau_2 \tau_3 \tau_1\tau_2}= \Delta^{\varpi_2}_{\tau_2}\Delta^{\varpi_2}_{\tau_3 \tau_1 \tau_2}- \Delta^{\varpi_2}_{\tau_3\tau_2}\Delta^{\varpi_2}_{ \tau_1 \tau_2} \pm  1,
    \end{equation}
    which is equivalent to $f_{\varpi_2} = \pm 1.$
In this case we have that 
$$
\Delta^{\varpi_2}_{\hw_{\varpi_{2k}}}(e)=\epsilon_{\varpi_{2k}}=-1.
$$

\end{example}

\subsection{A graphical representation of Theorem \ref{thm:fundamental identities sph mnors} for \texorpdfstring{$(\SL_{2n}, \Sp_{2n})$}{SL2n, Sp2n}}

Let $1 \leq k < n$ and $w \in W= S_{2n}$. In the notation of Example \ref{ex: sp sl pt 5}, recall that 
 $$R_k= \{ n-k+1, n-k+2, \dots, n , \overline{n}, \dots , \overline{n-k+2}, \overline{n-k+1} \}.$$
 Moreover, let 
 $$
 R_0= \emptyset \quad R_n= [2n].
 $$
By Corollary \ref{cor:sph minors only depend on weights}, the generalised spherical minor $\Delta^{\varpi_{2k}}_w$ is completely determined by the set $w(R_k)$.
Indeed
$$
w y \varpi_{2k}^* = - \sum_{j \in w(R_k) }\varepsilon_j.
$$
 A more precise statement is given in \eqref{eq: delta pi2k w sp sl}. In particular, in the present section we write 
$$
\Delta^{\varpi_{2k}}_w=:\Delta_{w(R_k)}.
$$ 
We set 
$$\Delta_{R_0}:= 1 =: \Delta_{R_n}.$$
Then, for $w=e$, relation \eqref{eq:fundamental iddentity sh minors} can be read as 
$$
\begin{array}{rl}
  \Delta_{R_k}\Delta_{\hw_{\varpi_{2k}}(R_k)} = & \Delta_{ \tau_{2(n-k)}(R_k)}\Delta_{ \tau_{2(n-k)}\hw_{\varpi_{2k}}(R_k)}  \\[0.5em]
   & - \Delta_{ \tau_{2(n-k)+1} \tau_{2(n-k)}(R_k)}\Delta_{ \tau_{2(n-k)+1} \tau_{2(n-k)} \hw_{\varpi_{2k}}(R_k)} \pm \Delta_{R_{k-1}} \Delta_{R_{k+1}}.
\end{array}$$

\bigskip

\noindent Note that, for any $v \in U_{\varpi_{2k}}$, we have that
 $$
 v(R_k) \cup v \hw_{\varpi_{2k}} (R_k)= R_{k+1} \quad \text{and} \quad v(R_k) \cap v \hw_{\varpi_{2k}} (R_k)= R_{k-1}.
 $$
The three monomials of the form $\Delta_{v(R_k)}\Delta_{v \hw_{\varpi_{2k}}(R_k)}$, for $v \in U_{\varpi_{2k}}$, can be interpreted as the set of pairs of the form $\{ R_{k-1} \cup I \, , \, R_{k+1} \setminus I\}$, where $I$ is a subset of cardinality two of $\{n-k,n-k+1, \overline{n-k}, \overline{n-k+1}\}$.

\bigskip
 
We chose to represent the minor $\Delta^{\varpi_{2k}}_w$ with a column vector of $2n$-elements. The entry in position $i$ is a blue star symbol ($\color{blue}{\bigstar}$) if $i \in w(R_k)$, and a dot symbol ($\cdot$) otherwise.
 For example, for $n=2$ and $k=1$, we have

 $$\Delta^{\varpi_2}_e= \bmat \cdot  \\
 \color{blue}{\bigstar} \\
 \color{blue}{\bigstar}\\
 \cdot \emat \quad \Delta^{\varpi_2}_{\tau_2}= \bmat \cdot  \\
\cdot  \\
 \color{blue}{\bigstar}\\
  \color{blue}{\bigstar} \emat.
 $$
 Moreover, the function $1 \in \CC[G]$ is denoted both with a column of $2n$ dots and with a column of $2n$ stars.
 Then, for $n=2$, $k=1$ and $w=e$, relation \eqref{eq:fundamental iddentity sh minors} becomes 
 
 \begin{equation}
     \label{eq:fund sph graphical n=2 k=1}
     \bmat \cdot  \\
 \color{blue}{\bigstar} \\
 \color{blue}{\bigstar}\\
 \cdot \emat \cdot 
 \bmat 
 \color{blue}{\bigstar} \\
 \cdot  \\
 \cdot  \\
 \color{blue}{\bigstar} \emat 
 %%%%%%%%%
 = 
 \bmat \cdot  \\
\cdot  \\
 \color{blue}{\bigstar}\\
  \color{blue}{\bigstar} \emat
  \cdot
  \bmat 
 \color{blue}{\bigstar}\\
  \color{blue}{\bigstar} \\
  \cdot  \\
\cdot  \emat 
%%%%%%%%%%
- 
 \bmat \cdot  \\
 \color{blue}{\bigstar}\\
\cdot  \\
  \color{blue}{\bigstar} \emat 
  \cdot
  \bmat 
 \color{blue}{\bigstar}\\
  \cdot  \\
   \color{blue}{\bigstar} \\
\cdot  \emat
%%%%%%%%5
\pm \bmat
\cdot \\
\cdot \\
\cdot \\
\cdot \emat
\cdot 
\bmat
\color{blue}{\bigstar}\\
\color{blue}{\bigstar}\\
\color{blue}{\bigstar}\\
\color{blue}{\bigstar} \emat.
 \end{equation}

\bigskip

For $n$, $k$ general and $w=e$, the graphical version of \eqref{eq:fundamental iddentity sh minors} can be obtained from \eqref{eq:fund sph graphical n=2 k=1} by applying the following two operations to any column vector of \eqref{eq:fund sph graphical n=2 k=1}.
\begin{enumerate}
    \item Insert $2(k-1)$ stars ($\color{blue}{\bigstar}$) in the middle of the column (that is between row 2 and 3).
    \item Insert $n-k-1$ dots ($\cdot$) in the upper part of the column (that is before the first entry) and $n-k-1$ dots ($\cdot$) in the lower part (that is below the last entry).
\end{enumerate}

For $n=3,k=2$ and $w=e$ we get
\begin{equation*}
     \bmat
     \cdot  \\
     \color{blue}{\bigstar}   \\
 \color{blue}{\bigstar} \\
 \color{blue}{\bigstar}\\
 \color{blue}{\bigstar}  \\
 \cdot  \emat \cdot 
 \bmat 
 \color{blue}{\bigstar} \\
 \cdot  \\
 \color{blue}{\bigstar}   \\
 \color{blue}{\bigstar}   \\
 \cdot \\
 \color{blue}{\bigstar}
 \emat   
 %%%%%%%%%
 = 
 \bmat 
 \cdot  \\
 \cdot  \\
\color{blue}{\bigstar}  \\
 \color{blue}{\bigstar}\\
  \color{blue}{\bigstar} \\
  \color{blue}{\bigstar}   \emat
  \cdot
  \bmat \color{blue}{\bigstar} \\
 \color{blue}{\bigstar}\\
  \color{blue}{\bigstar} \\
  \color{blue}{\bigstar}   \\
  \cdot\\
\cdot  \emat 
%%%%%%%%%%
- 
 \bmat
 \cdot\\
 \color{blue}{\bigstar}   \\
 \color{blue}{\bigstar}\\
 \color{blue}{\bigstar} \\
\cdot  \\
  \color{blue}{\bigstar}  
  \emat 
  \cdot
  \bmat 
 \color{blue}{\bigstar}\\
 \cdot \\
  \color{blue}{\bigstar}   \\
  \color{blue}{\bigstar} \\
   \color{blue}{\bigstar} \\
\cdot
\emat
%%%%%%%%5
\pm \bmat
\cdot\\
\cdot \\
\color{blue}{\bigstar}  \\
\color{blue}{\bigstar}  \\
\cdot \\
\cdot
\emat
\cdot 
\bmat
\color{blue}{\bigstar} \\
\color{blue}{\bigstar}\\
\color{blue}{\bigstar}\\
\color{blue}{\bigstar}\\
\color{blue}{\bigstar} \\
\color{blue}{\bigstar}  \emat.
 \end{equation*}

For $n,k$ and $w$ general, the graphical version of relation \eqref{eq:fundamental iddentity sh minors} is obtained from the corresponding one for $w=e$, by permuting the stars in each column according to $w$. For example, for $n=3$, $k=2$ and $w= \tau_4=(3,5)$, then \eqref{eq:fundamental iddentity sh minors} is represented as 

\begin{equation*}
     \bmat
     \cdot  \\
     \color{blue}{\bigstar}   \\
 \color{blue}{\bigstar} \\
 \color{blue}{\bigstar}\\
 \color{blue}{\bigstar}  \\
 \cdot  \emat \cdot 
 \bmat 
 \color{blue}{\bigstar} \\
 \cdot  \\
  \cdot\\
  \color{blue}{\bigstar}   \\
 \color{blue}{\bigstar}   \\
 \color{blue}{\bigstar}
 \emat   
 %%%%%%%%%
 = 
 \bmat 
 \cdot  \\
 \cdot  \\
\color{blue}{\bigstar}  \\
 \color{blue}{\bigstar}\\
  \color{blue}{\bigstar} \\
  \color{blue}{\bigstar}   \emat
  \cdot
  \bmat \color{blue}{\bigstar} \\
 \color{blue}{\bigstar}\\
  \cdot\\
   \color{blue}{\bigstar} \\
  \color{blue}{\bigstar}   \\
\cdot  \emat 
%%%%%%%%%%
- 
 \bmat
 \cdot\\
 \color{blue}{\bigstar}   \\
 \cdot\\
 \color{blue}{\bigstar}\\
 \color{blue}{\bigstar} \\
  \color{blue}{\bigstar}  
  \emat 
  \cdot
  \bmat 
 \color{blue}{\bigstar}\\
 \cdot \\
  \color{blue}{\bigstar}   \\
  \color{blue}{\bigstar} \\
   \color{blue}{\bigstar} \\
\cdot
\emat
%%%%%%%%5
\pm \bmat
\cdot\\
\cdot \\
\cdot\\
\color{blue}{\bigstar}  \\
\color{blue}{\bigstar}  \\
\cdot
\emat
\cdot 
\bmat
\color{blue}{\bigstar} \\
\color{blue}{\bigstar}\\
\color{blue}{\bigstar}\\
\color{blue}{\bigstar}\\
\color{blue}{\bigstar} \\
\color{blue}{\bigstar}  \emat.
 \end{equation*}

\subsection{ \texorpdfstring{$(\Spin_{2n}, \Spin_{2n-1})$}{Spin2n, Spin2n-1}}

We set 
$$ 
z:= s_1 s_2 \cdots s_{n-1}, \quad \text{therefore} \quad y=s_{n-1} \cdots s_1.
$$
The Cartan matrix $A$ is of size one by one and $A=(2)$. Indeed, the small root system is of type $A_1$. 
For $1 \leq i \leq n$, we denote 
$$ 
\delta_i:= y \alpha_i \quad s_i:=s_{\alpha_i} \quad \tau_i:= s_{\delta_i}.
$$
Moreover, we define 
$$
v_j:= \begin{cases}
    \tau_j \cdots \tau_1 & \text{if} \quad 1 \leq j \leq n-1\\
    e & \text{if}  \quad j=0.
\end{cases}
$$

\noindent Then, we set 

$$
\begin{array}{l}
\hw_{\varpi_1}:= \tau_1 \tau_2 \cdots \tau_{n-2} \, \tau_{n-1} \, \tau_n \tau_{n-2} \cdots \tau_2 \tau_1= s_n s_{n-1}\\[0.5em]
U_{\varpi_1}:= \{ v_i \, : \, 0 \leq i \leq n-1\}.
\end{array}$$

Note that $V(\varpi_1)$ is a $2n$-dimensional space with a quadratic form. This allows to identify the group $\Spin_{2n}$ as the double cover of the associated orthogonal group. Then, the relation given in Theorem \ref{thm:fundamental identities sph mnors} is equivalent to the fact that the action of $\Spin_{2n}$, on $V(\varpi_1)$, preserves the norm of the $\Spin_{2n-1}$-invariant vectors.

\bigskip

Clearly, we have that 
$$
\lambda_{\varpi_1}= 0.
$$
Hence, the set $U_{\varpi_{1}}$ stabilises $y \lambda_{\varpi_{1}}^*=0.$ Moreover, by Lemma \ref{lem: reflections sph rk min}, then 
$$
\hw_{\varpi_{1}}= s_n s_{n-1} = s_{\ha_{n-1}} \in \hW.
$$ 
We are now in position to prove claims 1 to 3. Note that $\varpi_{1}^*= \varpi_{1} $.

\begin{proof}[Proof of Claim 1.]
    Let $1 \leq i \leq n-1$ and $v \in U_{\varpi_{1}} \cup U_{\varpi_{1}} \cdot \hw_{\varpi_{1}}.$ We have that $v\inv \delta_i \in y \Phi^+$ unless $v \in \{v_i, v_{i-1} \hw_{\varpi_{1}}\}.$ Moreover, 
    $$
        \langle v_i y \varpi_1^* \, , \, \delta_i^\vee \rangle =  
        \langle s_i \dots s_1 \varepsilon_1 \, , \, \alpha_i^\vee \rangle  
        =
        \langle \varepsilon_{i+1} \, , \, \alpha_{i}^\vee \rangle 
        =  -1.
    $$
    Similarly, we compute that $\langle v_{i-1} \hw_{\varpi_{1}} y \varpi_1^* \, , \, \delta_i^\vee \rangle = -1.$
Then, by Corollary \ref{cor: Z delta sph minor.} and the Leibniz rule, we compute that 
$$
    Z_{\delta_i} \cdot f_{\varpi_1} =
    (-1)^{i-1} \bigl ( \Delta^{\varpi_{1}}_{v_{i-1}}  \Delta^{\varpi_{1}}_{\tau_i v_{i-1}\hw_{\varpi_{1}}} -  \Delta^{\varpi_{1}}_{\tau_iv_{i}}   \Delta^{\varpi_{1}}_{ v_{i}\hw_{\varpi_{1}}} \bigr)= 0.
$$
Indeed $\tau_i v_{i-1}=v_i$.

Similarly, we have that $v\inv \delta_n \in y \Phi^+$ unless $v \in \{v_{n-2} \hw_{\varpi_{1}} , v_{n-1} \hw_{\varpi_{1}}\}. $ Moreover, we easily compute that $$\langle v_{n-2} \hw_{\varpi_{1}} y \varpi_1^* \, , \, \delta_n^\vee \rangle = -1 = \langle v_{n-1} \hw_{\varpi_{1}} y \varpi_1^* \, , \, \delta_n^\vee \rangle.$$
Then, note that $\tau_n v_{n-1} \hw_{\varpi_{1}}= v_{n-2}$ and $\tau_n v_{n-2} \hw_{\varpi_{1}}= v_{n-1}$. Hence, from Corollary \ref{cor: Z delta sph minor.} and the Leibniz rule, we deduce that
$$
    Z_{\delta_n} \cdot f_{\varpi_1} =
    (-1)^{n-2} \bigl ( \Delta^{\varpi_{1}}_{v_{n-2}}  \Delta^{\varpi_{1}}_{\tau_n v_{n-2}\hw_{\varpi_{1}}} -  \Delta^{\varpi_{1}}_{v_{n-1}}   \Delta^{\varpi_{1}}_{ \tau_n v_{n-1}\hw_{\varpi_{1}}} \bigr)= 0.
$$
\end{proof}

\begin{proof}[Proof of Claim 2.]
     Let $u=s_1 \cdots s_{n-2} s_{n-1} s_n s_{n-2} \cdots s_1$. A direct computation shows that 
    $$u \varpi_1= - \varepsilon_1= -\varpi_1.$$
    It follows at ones that $$y \varpi_1^* + \hw_{\varpi_1} y \varpi_1^*= y \bigl ( \varpi_1 + u \varpi_1 \bigr) =0.$$
\end{proof}

\begin{proof}[Proof of Claim 3.]
   For $0<i<n$, we have that $$v_i y \varpi_1^*= y s_i \cdots s_1 \varepsilon_1= \varepsilon_i.$$
   We conclude by noticing that $\rho(\varepsilon_i)= \widehat \varepsilon_i \neq 0.$
\end{proof}

\subsection{ \texorpdfstring{$(E_{6}, F_4)$}{E6, F4}}
We set
$$
z:= s_6 s_1 s_3 s_4 s_5 s_4 s_3 s_2 s_4 s_3 s_5 s_1, \quad \text{therefore} \quad 
y=s_1 s_5 s_3 s_4 s_2 s_3 s_4 s_5 s_4 s_3 s_1 s_6.
$$
The small root system is of type $A_2$. The Cartan matrix $A$ is given by 
$$
a_{\varpi_i, \varpi_j}:= \begin{cases}
    2 & \text{if} \quad i = j\\
     -1 & \text{if} \quad i \neq j.
\end{cases}
$$ 
For $1 \leq i \leq 6$, we denote 
$$ \delta_i:= y \alpha_i \quad s_i:=s_{\alpha_i} \quad \tau_i:= s_{\delta_i}.
$$
Let $\sigma := ( 1, 6) (3,5) \in S_6$ be the permutation corresponding to the non-trivial symmetry of the Dynkin diagram of $E_6$.
With little abuse of notation, we also denote by $\sigma : \lg \longto \lg$ the authomorphism induced by 
$$
    Z_{\pm \delta_i}  \longmapsto   Z_{\pm \delta_{\sigma(i)}},
$$
The authomorphism $\sigma$ induces maps $X(T) \longto X(T)$ and $W \longto W$, that are again denote by $\sigma$. Note that $\sigma \bigl( \tau_i \bigr)= \tau_{\sigma(i)}$ and $\sigma \bigl( \delta_i\bigr) = \delta_{\sigma(i)}.$
Let 
$$
v_0:= e \quad v_1:= \tau_1 \quad v_2:= \tau_3 \tau_1 \quad v_3:= \tau_4 \tau_3 \tau_1 \quad v_4:= \tau_ 2 \tau_4 \tau_3 \tau_1
$$

\noindent Then, we set 

$$
\begin{array}{l c l }
   \hw_{\varpi_6}:= \tau_1 \tau_3 \tau_4 \tau_2 \tau_5 \tau_4 \tau_3 \tau_ 1 & & 
U_{\varpi_6}:= \{ v_i \, : \, 0 \leq i \leq 4\} \\[0.5em]
\hw_{\varpi_1}:= \sigma \bigl( \hw_{\varpi_{6}} \bigr) & & U_{\varpi_1}:= \sigma \bigl( U_{\varpi_{6}} \bigr).
\end{array}
$$
We complete the proof of Theorem \ref{thm:fundamental identities sph mnors} for $\varpi_6$. The same strategy of proof also works in the case of $\varpi_1$. Indeed, by means of $\sigma$, it's easy to see that the statements we prove for the $\varpi_6$ case, imply the analogues ones for the $\varpi_1$ case. 

\bigskip

Note that 
$$\lambda_{\varpi_6}= \varpi_1 \quad \lambda_{\varpi_{6}}^*= \varpi_6 \quad \varpi_6^*=\varpi_1.$$

It's clear that the set $U_{\varpi_{6}}$ stabilises $y \lambda_{\varpi_{6}}^*.$ Indeed, $s_1,s_2,s_3$ and $s_4$ stabilise $\varpi_6$. We can compute that 
$$
     \hw_{\varpi_{6}}=  y s_1s_3s_4s_2s_5s_4s_3s_1 y \inv 
     =  s_1 s_6 s_3 s_5 s_6 s_1 = s_{\ha_4} s_{\ha_3} s_{\ha_4}.
$$
The last equality is a consequence of Lemma \ref{lem: reflections sph rk min}. In particular, we have that $\hw_{\varpi_{6}} \in \hW.$ Hence, claims 1,2 and 3 imply Theorem \ref{thm:fundamental identities sph mnors}. Claim 2 and 3 follow by some elementary calculations.  We focus on Claim 1.

\begin{proof}[Proof of Claim 1.]
    Table \ref{tab:E6,F4}, which is obtained by direct computation, should be read as the analogue Table \ref{tab:sl sp}.

    \begin{table}[H]
        \centering
        \begin{tabular}{l|c|c|c|c|c|c|c|c | c | c}
             &  $e$ & $ \hw_{\varpi_{6}} $ & $v_1$ & $v_1 \hw_{\varpi_{6}}$ & $v_2$ & $v_2 \hw_{\varpi_{6}}$ & $v_3$ & $v_3 \hw_{\varpi_{6}}$ & $v_4$ & $v_4 \hw_{\varpi_{6}}$\\
             \hline 
          $\delta_1$   & $\color{Green}{+}$ & $\color{red}{-1}$ & $\color{red}{-1}$ & $\color{Green}{+}$ & $\color{Green}{+}$ &  $\color{Green}{+}$   & $\color{Green}{+}$ & $\color{Green}{+}$  & $\color{Green}{+}$ & $\color{Green}{+}$ \\
             \hline 
            $\delta_{3}$ & $\color{Green}{+}$ & $\color{Green}{+}$ & $\color{Green}{+}$ &$\color{red}{-1}$ &  $\color{red}{-1}$ & $\color{Green}{+}$ & $\color{Green}{+}$ &  $\color{Green}{+}$ & $\color{Green}{+}$ & $\color{Green}{+}$ \\
             \hline 
             $\delta_{4}$ & $\color{Green}{+}$&  $\color{Green}{+}$ &  $\color{Green}{+}$ & $\color{Green}{+}$ &  $\color{Green}{+}$  & $\color{red}{-1}$ & $\color{red}{-1}$ &  $\color{Green}{+}$ &  $\color{Green}{+}$ &  $\color{Green}{+}$ \\
             \hline 
             $\delta_{2}$ & $\color{Green}{+}$&  $\color{Green}{+}$ &  $\color{Green}{+}$ & $\color{Green}{+}$ &  $\color{Green}{+}$  &  $\color{Green}{+}$ &  $\color{Green}{+}$ &   $\color{red}{-1}$ &  $\color{red}{-1}$ & $\color{Green}{+}$\\
             \hline 
            $\delta_{5}$ & $\color{Green}{+}$&  $\color{Green}{+}$ &  $\color{Green}{+}$ & $\color{Green}{+}$ &  $\color{Green}{+}$  &  $\color{Green}{+}$ &  $\color{Green}{+}$ &   $\color{red}{-1}$ &  $\color{Green}{+}$ & $\color{red}{-1}$\\
             \hline 
             $\delta_{6}$ & $\color{Green}{+}$&  $\color{Green}{+}$ &$\color{Green}{+}$ & $\color{Green}{+}$ &  $\color{Green}{+}$  & $\color{Green}{+}$ & $\color{Green}{+}$ &  $\color{Green}{+}$ & $\color{Green}{+}$ & $\color{Green}{+}$\\
        \end{tabular}
        \caption{Infinitesimal action \texorpdfstring{$(E_{6}, F_{4})$}{E_6, F_4}}
        \label{tab:E6,F4}
    \end{table}
As in the case of $(\SL_{2n}, \Sp_{2n})$, Table \ref{tab:B3 G2} easily allows to conclude using Corollary \ref{cor: Z delta sph minor.} and the Leibniz rule. For example, using that $v_3= \tau_5 v_4 \hw_{\varpi_{6}}$ and $v_4= \tau_5 v_3 \hw_{\varpi_{6}}$, we compute that 

$$Z_{\delta_5} \cdot f_{\varpi_6} = (-1)^3 \bigl( \Delta_{v_3}^{\varpi_6} \Delta_{\tau_5 v_3 \hw_{\varpi_{6}}}^{\varpi_6} - \Delta_{v_4}^{\varpi_6} \Delta_{\tau_5 v_4 \hw_{\varpi_{6}}}^{\varpi_6} \bigr)=0.$$
\end{proof}

\subsection{ \texorpdfstring{$(B_{3}, G_2)$}{B3, G2}}
\label{sec:G2}
We set 
$$z:= s_3 s_2 s_3, \quad \text{therefore} \quad y=z.
$$
The Cartan matrix $A$ is of size one by one and $A=(2)$. Indeed, the small root system is of type $A_1$. 

For $1 \leq i \leq 3$, we denote 
$$ 
\delta_i:= y \alpha_i \quad s_i:=s_{\alpha_i} \quad \tau_i:= s_{\delta_i}.
$$
Let 
$$
v_0:= e \quad v_1:= \tau_3 \quad v_2:= \tau_2 \tau_3 \quad v_3:= \tau_3 \tau_2 \tau_3
$$

\noindent Then, we set 

$$
\hw_{\varpi_3}:=  \tau_3 \tau_2 \tau_3 \tau_1 \tau_2 \tau_3 \quad \quad
U_{\varpi_3}:=  \{ v_i \, : \, 0 \leq i \leq 3\}.
$$
Clearly, 
$$\lambda_{\varpi_3}= 0.$$
Obviously, the set $U_{\varpi_{3}}$ stabilises $y \lambda_{\varpi_{3}}^*=0.$ Moreover, note that 
$$
     \hw_{\varpi_{3}}=  y s_3s_2s_3s_1s_2s_3 y \inv 
     =  s_1 s_3 = s_{\ha_1}.
$$
The last equality is a consequence of Lemma \ref{lem: reflections sph rk min}. In particular, the element $\hw_{\varpi_{3}}$ belongs to $\hW.$ Hence, claim 1,2 and 3 imply Theorem \ref{thm:fundamental identities sph mnors}. Claim 2 and 3 follow by some elementary calculations.  We focus on Claim 1.

\begin{proof}[Proof of Claim 1.]
    Table \ref{tab:B3 G2}, which is obtained by direct computation, should be read as the analogue Table \ref{tab:sl sp}.

    \begin{table}[H]
        \centering
        \begin{tabular}{l|c|c|c|c|c|c|c|c}
             &  $e$ & $ \hw_{\varpi_{3}} $ & $v_1$ & $v_1 \hw_{\varpi_{3}}$ & $v_2$ & $v_2 \hw_{\varpi_{3}}$ & $v_3$ & $v_3 \hw_{\varpi_{3}}$\\
             \hline 
          $\delta_1$   & $\color{Green}{+}$ & \color{Green}{+} & \color{Green}{+} & $\color{Green}{+}$ & $\color{Green}{+}$ &  $\color{red}{-1}$   & $\color{Green}{+}$ & $\color{red}{-1}$ \\
             \hline 
            $\delta_{2}$ & $\color{Green}{+}$ & $\color{Green}{+}$ & $\color{Green}{+}$ & $\color{red}{-1}$ &  $\color{red}{-1}$ & $\color{Green}{+}$ &  $\color{Green}{+}$ &  $\color{Green}{+}$ \\
             \hline 
             $\delta_{3}$ & $\color{Green}{+}$&  $\color{red}{-1}$ & $\color{red}{-1}$ & $\color{Green}{+}$ &  $\color{Green}{+}$  & $\color{red}{-1}$ & $\color{red}{-1}$ &  $\color{Green}{+}$
        \end{tabular}
        \caption{Infinitesimal action \texorpdfstring{$(B_{3}, G_{2})$}{B_3,G_2}}
        \label{tab:B3 G2}
    \end{table}
As in the previous cases, Table \ref{tab:B3 G2} easily allows to conclude using Corollary \ref{cor: Z delta sph minor.} and the Leibniz rule.
\end{proof}

\bigskip

\textbf{Aknowledgements.} I would like to thank Jacopo Gandini for some useful discussion on spherical subgroups. This work was performed within the framework of the LABEX MILYON (ANR-10-
LABX-0070) of
Université de Lyon,
within the program "Investissements
d'Avenir" (ANR-11-IDEX-0007) operated by the French National Research Agency
(ANR).

\bibliographystyle{alpha}
%\bibliographystyle{spmpsci}
%\bibliographystyle{spbasic}

%\bibliography{tensorKacMoody}
\bibliography{biblio}

\end{document}